\documentclass[11pt]{article}
\usepackage[T1]{fontenc}
\usepackage{lmodern}
\usepackage{amsmath,amsthm,amssymb,amsfonts}
\usepackage{enumerate}
\usepackage{epsfig}
\usepackage[dvipsnames,svgnames,table]{xcolor}
\usepackage{pgf,tikz}
\usetikzlibrary{calc}
\usepackage{fullpage}
\usepackage[colorlinks=true,linkcolor=blue,urlcolor=blue,citecolor=red]{hyperref}
\usepackage{booktabs}
\usepackage[nameinlink]{cleveref}
\usepackage{enumitem}

\newtheorem{theorem}{Theorem}[section]
\newtheorem{proposition}[theorem]{Proposition}
\newtheorem{lemma}[theorem]{Lemma}
\newtheorem{corollary}[theorem]{Corollary}

\theoremstyle{definition}
\newtheorem{definition}[theorem]{Definition}

\theoremstyle{remark}
\newtheorem{remark}[theorem]{Remark}

\let\emph\relax 
\DeclareTextFontCommand{\emph}{\bfseries\em}

\definecolor{colR}{rgb}{.932,.172,.172}
\definecolor{colB}{rgb}{.255,.41,.884}
\definecolor{colG}{rgb}{0,0.7,0}

\tikzstyle{edge}=[line width=1.5pt]
\tikzstyle{dedge}=[edge,dashed,gray]
\tikzstyle{redge}=[edge,colR]
\tikzstyle{bedge}=[edge,colB]
\tikzstyle{gedge}=[edge,colG]
\tikzstyle{lnode}=[circle,white,draw, fill=black,inner sep=1pt, font=\scriptsize]
\tikzstyle{hollow}=[circle,gray,draw, thick, fill=white,inner sep=0pt, minimum size=4pt]

\colorlet{colfg}{black}
\colorlet{colgraphv}{colfg!75!white}
\colorlet{colgraphe}{colfg!55!white}
\tikzstyle{vertex}=[fill=colgraphv,circle,inner sep=0pt, minimum size=4pt]
\tikzstyle{edge}=[line width=1.5pt,colgraphe]

\title{How many contacts can exist between oriented squares of various sizes?}
\author{Sean Dewar\thanks{School of Mathematics, University of Bristol. E-mail: \texttt{sean.dewar@bristol.ac.uk}}
}

\begin{document}
\date{}
\maketitle

\begin{abstract}
A homothetic packing of squares is any set of various-size squares with the same orientation where no two squares have overlapping interiors.
	If all $n$ squares have the same size then we can have up to roughly $4n$ contacts by arranging the squares in a grid formation.
	The maximum possible number of contacts for a set of $n$ squares will drop drastically, however, if the size of each square is chosen more-or-less randomly.
	In the following paper we describe a necessary and sufficient condition for determining if a set of $n$ squares with fixed sizes can be arranged into a homothetic square packing with more than $2n-2$ contacts.
	Using this, we then prove that any (possibly not homothetic) packing of $n$ squares will have at most $2n-2$ face-to-face contacts if the various widths of the squares do not satisfy a finite set of linear equations.
\end{abstract}



\section{Introduction}\label{sec:intro}

Throughout the paper we fix $S := \{ (x,y) : -1 \leq x,y \leq 1\}$ to be the standard square and $[n] := \{1,\ldots,n\}$ to be the first $n$ positive integers.
A \emph{homothetic copy} of a set $A \subset \mathbb{R}^d$ is any set 
\begin{align*}
	r A +p := \{ r x + p : x \in A\},
\end{align*}
for some scalar $r>0$ and some point $p \in \mathbb{R}^d$.
With this, we define a \emph{homothetic packing of $n$ squares}, or \emph{homothetic square packing} for short, to be any set $P = \{S_1,\ldots,S_n\}$ of homothetic copies of $S$ where for each distinct pair $i,j \in [n]$,
the interiors $S_i^\circ$ and $S_j^\circ$ of the sets $S_i$ and $S_j$ respectively are disjoint.
Given each square in $P$ is of the form $S_i = r_i S + p_i$,
we can characterise $P$ uniquely by two types of variables:
the positive scalar \emph{radii} $r_1,\ldots, r_n$
and the 2-dimensional real vector \emph{centres} $p_1,\ldots,p_n$.
The  \emph{contact graph} $G=([n],E)$ of $P$ is the (simple) graph where $\{i,j\} \in E$ if and only if $i \neq j$ and $S_i \cap S_j \neq \emptyset$.
See \Cref{fig:homsqpacki} for an example of a homothetic square packing and its corresponding contact graph.

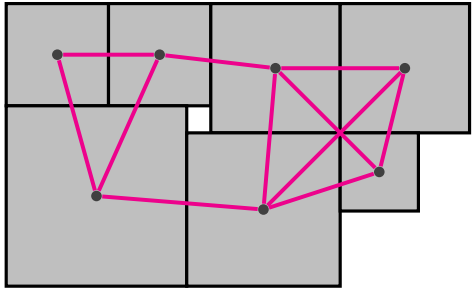
\begin{figure}[ht]
\begin{center}
\begin{tikzpicture}[scale=0.4]
		\draw[very thick, fill=lightgray] (-6,1) rectangle (-2.6,4.4);
		\draw[very thick, fill=lightgray] (-2.6,1) rectangle (0.8,4.4);
		\draw[very thick, fill=lightgray] (0.8,0.1) rectangle (5.1,4.4);
		\draw[very thick, fill=lightgray] (-6,-5) rectangle (0,1);
		\draw[very thick, fill=lightgray] (0,-5) rectangle (5.1,0.1);
		\draw[very thick, fill=lightgray] (5.1,0.1) rectangle (9.4,4.4);
		\draw[very thick, fill=lightgray] (5.1,-2.5) rectangle (7.7,0.1);
		
		\node[vertex] (1) at (-4.3,2.7) {};
		\node[vertex] (2) at (-0.9,2.7) {};
		\node[vertex] (3) at (2.95,2.25) {};
		\node[vertex] (4) at (-3,-2) {};
		\node[vertex] (5) at (2.55,-2.45) {};
		
		\node[vertex] (E) at (7.25,2.25) {};
		\node[vertex] (F) at (6.4,-1.2) {};
			
		\draw[edge, magenta] (1)edge(2);
		\draw[edge, magenta] (2)edge(3);
		\draw[edge, magenta] (2)edge(4);
		\draw[edge, magenta] (1)edge(4);
		\draw[edge, magenta] (3)edge(5);
		\draw[edge, magenta] (4)edge(5);
		\draw[edge, magenta] (E)edge(3);
		\draw[edge, magenta] (E)edge(5);
		\draw[edge, magenta] (F)edge(3);
		\draw[edge, magenta] (F)edge(5);
		\draw[edge, magenta] (E)edge(F);
	\end{tikzpicture}
	\end{center}
	\caption{A homothetic square packing with 11 contacts. The edges of its contact graph are represented by the coloured lines.}
	\label{fig:homsqpacki}
\end{figure}

An immediate question one can ask is the following: what is the upper bound on the number of contacts for a homothetic packing of $n$ squares?
It is easy to see that roughly $4n - 6\sqrt{n} +2$ contacts can be achieved;
given $n = n_1 n_2$ squares with the same radii, the homothetic square packing formed by arranging the squares in a $n_1 \times n_2$ grid has $4n - 3(n_1 + n_2) + 2$ contacts.
However it is easy to see that this type of packing (or any similar packing formed by replacing square blocks of $k^2$ unit squares with a single square of radii $k$)
forces the radii to satisfy some rational linear constraints.

What, then, should the maximum number of contacts be for a homothetic square packing if the radii are picked more-or-less randomly?
Although our previous construction proves that around $4n$ contacts are indeed possible, one very quickly notices that this is not the case if the radii of the squares are chosen at random.
We encourage the reader now to construct $n$ squares of various widths (whether from paper or other means) and try to arrange them in a way that maximises the amount of contacts while keeping the squares oriented in the same way.
It becomes apparent very quickly that the most amount of contacts possible is never more than $2n-2$, no matter how the squares are arranged.

Our main result of this paper is that the maximum number of contacts achievable by a homothetic packing of $n$ squares will exceed $2n-2$ if and only if the radii of the chosen squares satisfy some very basic linear constraints.

\begin{theorem}\label{mainthm}
    Let $r_1,\ldots,r_n$ be positive scalars.
    Then the following statements are equivalent:
    \begin{enumerate}
        \item\label{mainthm1} Every homothetic packing of $n$ squares with radii $r_1,\ldots, r_n$ has at most $2n-2$ contacts.
        \item\label{mainthm2} The only function $\sigma : [n] \rightarrow \{-1,0,1\}$ with at least 4 zeroes that satisfies the equation $\sum_{i=1}^n \sigma_i r_i = 0$ is the zero function.
    \end{enumerate}
\end{theorem}

In \cite{congortheran2019},
Connelly, Gortler and Theran proved an analogous result to \Cref{mainthm} for disc packings (the definition of a disc packing being identical to that of a homothetic square packing except with the square $S$ replaced by the closed unit disc).

\begin{theorem}[\cite{congortheran2019}]\label{t:cgt}
    Let $r_1,\ldots,r_n$ be positive scalars that are mutually distinct and form an algebraically independent set.
    Then every packing of $n$ discs with radii $r_1,\ldots, r_n$ has at most $2n-3$ contacts.
\end{theorem}

Although similar, \Cref{mainthm,t:cgt} do differ in a few specific ways.
Connelly, Gortler and Theran proved \Cref{t:cgt} by constructing a smooth manifold of disc packings with a given contact graph,
and then showing that any disc packing with algebraically independent radii will be a regular point of a projection.
Our method for homothetic square packings, however, only requires very simple concepts from geometry and combinatorics.
Furthermore, \Cref{mainthm} describes both a sufficient and necessary condition for a set of radii to generate packings with a low amount of contacts,
while \Cref{t:cgt} only provides a sufficient condition.

\Cref{t:cgt} was in recent years also extended to homothetic packings of any convex body $C \subset \mathbb{R}^2$ (a compact convex set with non-empty interior), so long as the convex body is also centrally symmetric ($x \in C$ if and only if $-x \in C$), strictly convex (every point on the boundary of $C$ is contained in a supporting hyperplane of $C$ that intersects $C$ at exactly one point) and smooth (every point on the boundary of $C$ is contained in exactly one supporting hyperplane of $C$).
Any such set is also known as a \emph{regular symmetric body}.

\begin{theorem}[\cite{dew21}]\label{t:dew}
    For every regular symmetric body $C$ and every positive integer $n \in \mathbb{N}$,
    there exists a conull\footnote{A set is conull if its complement is a null set, i.e., has Lebesgue measure zero.} set of vectors $(r_1,\ldots,r_n)$ in $\mathbb{R}^n_{>0}$ so that the following holds:
    every packing of $n$ homothetic copies of $C$ with radii $r_1,\ldots, r_n$ has at most $2n-2$ contacts.
\end{theorem}

Although a square is a convex body, it is neither strictly convex nor smooth,
and hence is not covered by \Cref{t:dew}.
In any case, \Cref{t:dew} is a noticeably weaker result than both \Cref{mainthm,t:cgt} as it does not describe either a necessary or a sufficient condition for a given set of radii to only generate packings with low numbers of contacts.

The paper is structured as follows.
In \Cref{sec:basic} we introduce the various types of contacts a homothetic square packing can have,
and use this to define red and blue edge colourings for a homothetic square packing's contact graph.
In \Cref{sec:cycle} we investigate the effect the weak generic condition (see \Cref{def:wgc}) has on the cycles in the induced red and blue subgraphs of the contact graph.
These techniques are then applied to proving \Cref{mainthm} in \Cref{sec:main}.
In \Cref{sec:facetoface} we use \Cref{mainthm} to obtain analogous results for square packings that allow for rotated squares (\Cref{cor:rotatedsquares}).
We conclude the paper in \Cref{sec:cube} by proving that the natural analogue of \Cref{mainthm} cannot be extended to homothetic cube packings

\begin{remark}
	A homothetic square packing can be considered to be a bar-and-joint framework in the normed space $(\mathbb{R}^2,\| \cdot\|_\infty)$ by modelling the centres as points in $\mathbb{R}^2$, the edges as $\ell_\infty$-norm distance equalities between points (to simulate squares in contact), and the non-edges as strict $\ell_\infty$-norm distance inequalites (to simulate squares not intersecting).
	Whilst we will avoid using the language of bar-and-joint framework rigidity theory here,
	we do direct interested readers to the work of Kitson and Power for more information about the topic \cite{kit-pow}.
\end{remark}

\section{Contact graphs for homothetic square packings}\label{sec:basic}

Unless stated otherwise,
a homothetic square packing $P=\{S_1,\ldots, S_n\}$ will have contact graph $G=([n],E)$, centres $p_1,\ldots, p_n$ and radii $r_1,\ldots,r_n$.
We also denote the $x$- and $y$-coordinates of a centre $p_i$ to be $x_i,y_i$,
i.e., $p_i = (x_i,y_i)$.
If two distinct squares $S_i$ and $S_j$ are in contact, at least one of two possible cases holds.
\begin{enumerate}
	\item If $r_i + r_j = |x_i - x_j| \geq |y_i - y_j|$ then $S_i$ and $S_j$ have an \emph{$x$-direction contact};
	equivalently,
	$S_i$ and $S_j$ have an $x$-direction contact if and only if $S_i \cap S_j = \{(a,t) \in \mathbb{R}^2 : b \leq t \leq c\}$ for some $a,b,c \in \mathbb{R}$ with $b \leq c$.
	\item If $r_i + r_j =|y_i - y_j| \geq |x_i - x_j|$ then $S_i$ and $S_j$ have a \emph{$y$-direction contact};
	equivalently,
	$S_i$ and $S_j$ have a $y$-direction contact if and only if $S_i \cap S_j = \{(t,a) \in \mathbb{R}^2 : b \leq t \leq c\}$ for some $a,b,c \in \mathbb{R}$ with $b\leq c$.
\end{enumerate}
The intersection of $S_i$ and $S_j$ will always contain the point
\begin{align}\label{eq:midpoint}
	p_{ij} := \left( x_i - \frac{r_i(x_i-x_j)}{r_i+r_j} , ~ y_i - \frac{r_i(y_i-y_j)}{r_i+r_j}\right) = \left( x_j + \frac{r_j(x_i-x_j)}{r_i+r_j} , ~ y_j + \frac{r_j(y_i-y_j)}{r_i+r_j} \right).
\end{align}
If two squares have both an $x$-direction and $y$-direction contact then $S_i \cap S_j = \{p_{ij}\}$,
and $p_{ij}$ will be a corner of both of $S_i$ and $S_j$.
In fact, this is the only way two squares in a homothetic packing can intersect at a single point.
See \Cref{fig:contacts} to see a diagram of the possible types of contact between two squares.

\begin{figure}[ht]
\begin{center}
\begin{tikzpicture}[scale=0.4]
		
		\draw[very thick, fill=lightgray] (0,-5) rectangle (5.1,0.1);
		\node[vertex] (5) at (2.55,-2.45) {};
		
		\draw[very thick, fill=lightgray] (5.1,-1.9) rectangle (9.2,2.2);
		\node[vertex] (E) at (7.23,0.23) {};

		\draw[edge, red] (E)edge(5);
		
	\end{tikzpicture}\qquad\qquad
    \begin{tikzpicture}[scale=0.4]
		
		\draw[very thick, fill=lightgray] (0.65,0.1) rectangle (4.95,4.4);
		\node[vertex] (3) at (2.8,2.25) {};
		
		\draw[very thick, fill=lightgray] (0,-5) rectangle (5.1,0.1);
		\node[vertex] (5) at (2.55,-2.45) {};

		\draw[edge, blue] (3)edge(5);
		
	\end{tikzpicture}\qquad\qquad
    \begin{tikzpicture}[scale=0.4]
		
		\draw[very thick, fill=lightgray] (0,-5) rectangle (5.1,0.1);
		\node[vertex] (5) at (2.55,-2.45) {};
		
		\draw[very thick, fill=lightgray] (5.1,0.1) rectangle (9.2,4.2);
		\node[vertex] (E) at (7.23,2.23) {};
		
		\path[edge,red] (E) edge [bend right=10] node {} (5);
		\path[edge,blue] (E) edge [bend left=10] node {} (5);
		
	\end{tikzpicture}
	\end{center}
	\caption{Three possible types of contact between two squares: an $x$-direction contact represented by a red edge (left), a $y$-direction contact represented by a blue edge (middle) and a $x$- and $y$-direction contact represented by a red-blue pair of parallel edges (right).
	Although the latter type of contact is represented by two parallel edges, it will still be considered to be a single contact.}
	\label{fig:contacts}
\end{figure}
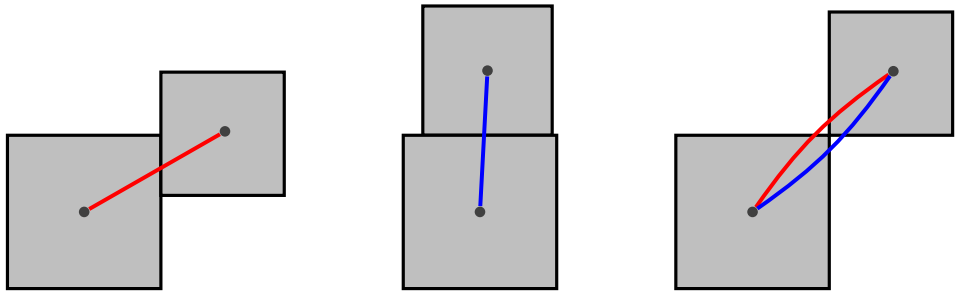

Using this extra information about the edges of the contact graph,
we now define $E_x, E_y \subset E$ to be the sets of $x$-direction and $y$-direction edges respectively.
With this notation, we have that $E_x \cup E_y = E$ and $E_x \cap E_y$ is exactly the set of contacts with a single point in the intersection.
When drawing the contact graph of a homothetic square packing, we shall always represent the edges in the set $E_x \setminus E_y$ by a red line, the edges in the set $E_y \setminus E_x$ by a blue line,
and the edges in the set $E_x \cap E_y$ by both a red line and a blue line.
Importantly, these ``double edges'' are still only counted as a single edge in our contact graph.
See \Cref{fig:contacts} to see how the different contacts are represented.
Interestingly, the subgraphs $([n],E_x),([n],E_y)$ must always be triangle-free.

\begin{lemma}\label{mainlem0}
	Let $P$ be a homothetic packing of $n$ squares.
	Then the coloured subgraphs $([n],E_x)$, $([n],E_y)$ of the contact graph $G=([n],E)$ are triangle-free.
\end{lemma}

\begin{proof}
	It suffices to prove that $([n],E_x)$ is triangle-free since rotating $P$ by $90^\circ$ will switch the edges $E_x$ and $E_y$.
	Suppose for contradiction that $([n],E_x)$ contains a triangle.
	By relabelling vertices of $G$ we may suppose that $\{1,2,3\}$ is a clique in $([n],E_x)$ and $x_1 \leq x_2 \leq x_3$.
	Since $r_i + r_j = |x_i-x_j|$ for $1 \leq i < j \leq 3$,
	all three $x_1,x_2,x_3$ must be distinct,
	i.e., $x_1<x_2 <x_3$.
	Hence
	\begin{align*}
		x_2 - x_1 = r_1 + r_2, \qquad x_3 - x_1 = r_1 + r_3, \qquad x_3 - x_2 = r_2 + r_3.
	\end{align*}
	By summing all three equations and halving the result,
	we have that $x_3 - x_1 = r_1 + r_2 + r_3$.
	However this now implies that $r_2 = 0$,
	contradicting that all radii are positive.
\end{proof}

One special way that one of the graphs $([n],E_x)$ or $([n],E_y)$ can contain a cycle is for four squares to share an intersection as seen in \Cref{fig:4share};
if this occurs, we say the four squares \emph{share a corner}.
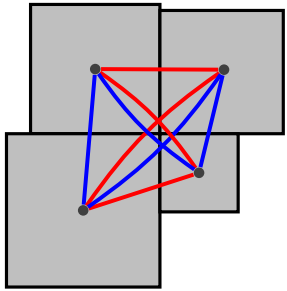
\begin{figure}[ht]
\begin{center}
\begin{tikzpicture}[scale=0.4]
		
		\draw[very thick, fill=lightgray] (0.8,0.1) rectangle (5.1,4.4);
		\node[vertex] (3) at (2.95,2.25) {};
		
		\draw[very thick, fill=lightgray] (0,-5) rectangle (5.1,0.1);
		\node[vertex] (5) at (2.55,-2.45) {};
		
		\draw[very thick, fill=lightgray] (5.1,0.1) rectangle (9.2,4.2);
		\node[vertex] (E) at (7.23,2.23) {};
		
		\draw[very thick, fill=lightgray] (5.1,-2.5) rectangle (7.7,0.1);
		\node[vertex] (F) at (6.4,-1.2) {};

		\draw[edge, blue] (3)edge(5);
		\draw[edge, red] (E)edge(3);
		\draw[edge, red] (F)edge(5);
		\draw[edge, blue] (E)edge(F);
		
		\path[edge,red] (E) edge [bend right=10] node {} (5);
		\path[edge,blue] (E) edge [bend left=10] node {} (5);
		
		\path[edge,red] (F) edge [bend right=10] node {} (3);
		\path[edge,blue] (F) edge [bend left=10] node {} (3);
		
	\end{tikzpicture}
	\end{center}
	\caption{Four squares sharing a corner.}
	\label{fig:4share}
\end{figure}

As we shall soon prove, 
the only way to generate cliques with more than three vertices is with four squares sharing a corner.
We first need to cover the following famous result of Helly.

\begin{theorem}[Helly's theorem; see, for example, \cite{dankleegru}]\label{t:helly}
    Let $\mathcal{C} = \{C_1,\ldots, C_n\}$ be a set of convex sets in $\mathbb{R}^d$ where $n \geq d+1$.
    If every $d+1$ distinct sets in $\mathcal{C}$ have a non-empty intersection,
    then $\bigcap_{i=1}^n C_i \neq \emptyset$.
\end{theorem}

We require the following two special cases of \Cref{t:helly} where the set $\mathcal{C}$ contains only homothetic copies of the standard square $S$.

\begin{lemma}\label{l:sqhelly}
    Let $\mathcal{C} = \{S_1,\ldots, S_n\}$ be a set of pairwise-intersecting homothetic copies of $S$.
    Then $\bigcap_{i=1}^n S_i \neq \emptyset$.
\end{lemma}

\begin{proof}	
	Define $\pi_x, \pi_y : \mathbb{R}^2 \rightarrow \mathbb{R}$ to be the linear projections where $\pi_x(x,y) = x$ and $\pi_y(x,y) = y$ for each point $(x,y) \in \mathbb{R}^2$.
	Since the sets in $\mathcal{C}$ are pairwise-intersecting,
	so too are the sets in both $\{\pi_x(S_1),\ldots,\pi_x(S_n)\}$ and $\{\pi_y(S_1),\ldots,\pi_y(S_n)\}$.
	By \Cref{t:helly},
	there exists points $x',y' \in \mathbb{R}$ such that $x' \in \bigcap_{i=1}^n \pi_x(S_i)$ and $y' \in \bigcap_{i=1}^n \pi_y(S_i)$.
	Equivalently,
	given each square $S_i$ is of the form $[a_i,b_i] \times [c_i,d_i]$,
	we have $a_i \leq x' \leq b_i$ and $c_i \leq y' \leq d_i$.
	Hence $(x',y') \in S_i$ for each $i \in [n]$.
\end{proof}

\begin{lemma}\label{l:3squares}
	Let $P$ be a homothetic packing of $n$ squares.
	If $\{i,j,k\}$ is a clique in the contact graph $G=([n],E)$,
	then $S_i \cap S_j \cap S_k$ contains exactly one point.
\end{lemma}

\begin{proof}
	By relabelling vertices we may assume that $\{1,2,3\}$ is the clique in $G$.	
	By \Cref{l:sqhelly},
	the set $S_1 \cap S_2 \cap S_3$ is non-empty,
	hence it is sufficient to prove that $S_1 \cap S_2 \cap S_3$ contains at most one point.
	If any of the sets $S_1 \cap S_2$, $S_1 \cap S_3$ or $S_2 \cap S_3$ contain exactly one point then $S_1 \cap S_2 \cap S_3$ contains at most one point.
	Suppose instead that all three sets $S_1 \cap S_2$, $S_1 \cap S_3$ and $S_2 \cap S_3$ contain more than one point.
	Then each distinct pair $S_i,S_j$ has either an $x$- or $y$-direction contact, but not both.
	By \Cref{mainlem0},
	one of these distinct pairs has an $x$-direction contact and another distinct pair has a $y$-direction contact.
	As the intersection of two perpendicular line segments is either an empty set or a single point,
	the set $S_1 \cap S_2 \cap S_3$ contains at most one point.
\end{proof}

The previous two lemmas allow us to characterise the cliques in the contact graph of any given homothetic square packing.

\begin{lemma}\label{l:4share}
	Let $P$ be a homothetic packing of $n$ squares.
	Then any clique of the contact graph $G=([n],E)$ has size at most 4,
	and any clique of size 4 will correspond to 4 squares sharing a corner.
\end{lemma}

\begin{proof}
	By relabelling the vertices,
	suppose that $\{1,\ldots,k\}$ is a clique of $G$ with $k \geq 4$.
	By \Cref{l:sqhelly,l:3squares},
	there exists a unique point $z$ in the intersections of the squares $S_1,\ldots,S_k$.
	Furthermore,
	$z$ must lie on the boundary of each square $S_1,\ldots,S_k$ by our assumption that $P$ is a homothetic square packing.
	The interior of each square $S_i$ covers an angle $\alpha_i$ of points around $z$.
	For each $i \in [k]$,
	we either have $\alpha_i = \pi$ and $z$ lies on exactly one face of $S_i$,
	or $\alpha_i = \pi/2$ and $z$ is a corner of $S_i$.
	As $\sum_{i=1}^k \alpha_i \leq 2 \pi$,
	we see that $k \leq 4$,
	with equality if and only if $S_1,S_2,S_3,S_4$ share a corner.
\end{proof}

Interestingly, no two distinct cliques of size 4 in the contact graph can share three vertices.

\begin{lemma}\label{l:cliques2share}
	Let $P$ be a homothetic packing of $n$ squares.
	If $K,K'$ are two distinct cliques of size 4 in $G=([n],E)$,
	then $|K \cap K'| \leq 2$.
\end{lemma}

\begin{proof}
	By relabelling the vertices of $G$,
	we may suppose that $K = \{1,2,3,4\}$ and $K' = \{1,2,3,5\}$.
	As the sets $\{1,2,3,4\}$ and $\{1,2,3,5\}$ are cliques in $G$,
	the squares $S_1,S_2,S_3,S_4$ intersect at a point $z$ and the squares $S_1,S_2,S_3,S_5$ intersect at a point $z'$.
	Since $z , z' \in S_1 \cap S_2 \cap S_3$,
	it follows from \Cref{l:3squares} that $z = z'$.
	However this implies that $z \in S_4 \cap S_5$,
	contradicting that $G$ can contain no cliques of size 5 (\Cref{l:4share}).
\end{proof}

Before we can deduce more properties of the contact graph, we require the following technical result regarding the straight-line embedding of the contact graph;
i.e., the mapping of $G$ into the plane where a vertex $i$ is considered to be the point $p_i$ and an edge is considered to be the closed line segment
\begin{equation*}
	[p_i,p_j] :=\big\{ tp_i + (1-t)p_j : t \in [0,1] \big\}.
\end{equation*}

\begin{lemma}\label{l:lineseg}
	Let $P$ be a homothetic packing of $n$ squares.
	Choose any edge $\{i,j\} \in E$ and let $p_{ij} = (x_{ij},y_{ij})$ be the point described in \cref{eq:midpoint}.
	Then the following holds.
	\begin{enumerate}
		\item \label{l:lineseg1} The closed line segment $[p_i,p_j]$ is contained in $S_i \cup S_j$.
		\item \label{l:lineseg2} The set $[p_i,p_j ] \setminus \{p_{ij}\}$ is contained in the set  $S_i^\circ \cup S_j^\circ$ (and hence in the interior of $S_i \cup S_j$).
		\item \label{l:lineseg3} The point $p_{ij}$ lies in the interior of the set $S_i \cup S_j$ if and only if $\{i,j\} \notin E_x \cap E_y$.
	\end{enumerate}
\end{lemma}

\begin{proof}
	Fix $p_i=(x_i,y_i)$ and $p_j=(x_j,y_j)$.
	Given a point $z = (x,y) \in [p_i,p_{ij}]$,
	we note that $|x - x_i| \leq r_i$ and $|y - y_i| \leq r_i$,
	with equality in one of these inequalities if and only if $z = p_{ij}$.
	An analogous observation can be made for any point in the line segment $[p_j,p_{ij}]$,
	hence \ref{l:lineseg1} and \ref{l:lineseg2} hold.
	It now suffices for us to check whether $p_{ij}$ lies in $(S_i \cup S_j)^\circ$.
	
	First suppose that $\{i,j\} \in E_x \cap E_y$.
	By rotating $P$ we may suppose that $x_i<x_j$ and $y_i<y_j$.
	For any $t > 0$, the point $p_{ij} + (t,-t)$ is not contained in $S_i \cup S_j$ as
	\begin{align*}
		|x_{ij} + t - x_i| = r_i + t > r_i  \qquad \text{and} \qquad |y_{ij} - t - y_j| = r_j + t > r_j.
	\end{align*}
	Hence $p_{ij}$ does not lie in $(S_i \cup S_j)^\circ$.
	
	Now suppose, without loss of generality, that $\{i,j\} \in E_x \setminus E_y$ and $x_i < x_j$.
	Fix 
	\begin{align*}
		\varepsilon := \min \left\{ r_i - \frac{r_i|y_i - y_j|}{r_i+r_j}, ~  r_j - \frac{r_j|y_i - y_j|}{r_i+r_j} \right\} > 0.
	\end{align*}
	Choose any $s \in (x_i - r_i, x_j + r_j)$ and $t \in (-\varepsilon,\varepsilon)$.
	Then the point $p_{ij} + (s,t)$ lies in $S_i \cup S_j$ as $x_i + r_i = x_j - r_j$ and
	\begin{align*}
		|y_{ij} + t - y_i| &\leq |t| + \left| \frac{r_i(y_i - y_j)}{r_i + r_j} \right| \leq r_i, \\
		|y_{ij} + t - y_j| &\leq |t| + \left| \frac{r_j(y_i - y_j)}{r_i + r_j} \right| \leq r_j.
	\end{align*}
	As this holds for any choice of $s,t$,
	the point $p_{ij}$ lies in $(S_i \cup S_j)^\circ$.
\end{proof}

If four squares do share a corner,
then the straight-line embedding of $G$ given by the centres $p_1,\ldots,p_n$ will not be planar. Fortunately, this is the only way that planarity can be lost.

\begin{lemma}\label{l:4squares}
	Let $P$ be a homothetic packing of $n$ squares.
	Then the following properties hold for any pair of edges $\{i,j\},\{k,\ell\} \in E$ that share no vertices.
	\begin{enumerate}
		\item \label{l:4squares1} The interiors of the sets $S_i \cup S_j$ and $S_k \cup S_\ell$ are disjoint.
		\item \label{l:4squares3} If the closed line segments $[p_i,p_j]$ and $[p_k,p_\ell]$ intersect, 
		then the squares $S_i,S_j,S_k,S_\ell$ share a corner, the edges $\{i,j\},\{k,\ell\}$ lie in both $E_x$ and $E_y$,
		and the edges $\{i,k\},\{i,\ell\},\{j,k\},\{j,\ell\}$ lie in the symmetric difference of $E_x$ and $E_y$ (denoted by $E_x \triangle E_y$).
	\end{enumerate}
\end{lemma}

\begin{proof}
	\ref{l:4squares1}:
	Suppose for contradiction that the interiors of the sets $S_i \cup S_j$ and $S_k \cup S_\ell$ are not disjoint.
	Let $R_{ij}$ and $R_{k\ell}$ be the relative interiors of the convex sets $S_i \cap S_j$ and $S_k \cap S_\ell$ respectively.
	Then $(S_i \cup S_j)^\circ = S_i^\circ \cup S_j^\circ \cup R_{ij}$ and $(S_k \cup S_\ell)^\circ = S_k^\circ \cup S_\ell^\circ \cup R_{k\ell}$.
	Since the interiors of the squares in $P$ are pairwise disjoint,
	the sets $S_i^\circ \cup S_j^\circ$ and $S_k^\circ \cup S_\ell^\circ$ must be disjoint.
	Hence we have, without loss of generality, that $R_{ij} \neq \emptyset$ and $R_{ij} \cap (S_k \cup S_\ell)^\circ \neq \emptyset$.
	By rotating and translating $P$ if necessary,
	we may further assume that $R_{ij} = \{(s,0) \in \mathbb{R}^2 : a < s <b \}$ for some $a,b \in \mathbb{R}$ with $a < b$.
	Choose a sufficiently small scalar $\varepsilon >0$ such that the open neighbourhood
	\begin{align*}
		R_{ij}^\varepsilon := \left\{ (s,t) \in \mathbb{R}^2 : ~ a < s <b, ~ -\varepsilon < t < \varepsilon \right\}
	\end{align*}
	of $R_{ij}$ is contained in $(S_i \cup S_j)^\circ$.
	Since $R_{ij}$ intersects $(S_k \cup S_\ell)^\circ$ non-trivially and non-empty open sets always have positive area,\footnote{As all sets in the plane mentioned throughout the paper are Lebesgue-measurable, we define the area of a set by using the Lebesgue measure on $\mathbb{R}^2$.}
	the intersection of the open sets $R_{ij}^\varepsilon$ and $(S_k \cup S_\ell)^\circ$ is a non-empty open set with positive area.
	Since both $R_{ij}$ and $R_{k\ell}$ are null sets (i.e., have zero area),
	it follows that $R^\varepsilon_{ij} \setminus R_{ij}$ (and hence $S_i^\circ \cup S_j^\circ$) intersects $S_k^\circ \cup S_\ell^\circ$ non-trivially,
	thus forcing a contradiction.
	
	\ref{l:4squares3}:
	By \Cref{l:lineseg}\ref{l:lineseg1} we have that $[p_i,p_j] \subset S_i \cup S_j$ and $[p_k,p_\ell] \subset S_k \cup S_\ell$.
	By \Cref{l:lineseg}\ref{l:lineseg2} and \Cref{l:lineseg}\ref{l:lineseg3},
	$\{i,j\} \in E_x \triangle E_y$ if and only if $[p_i,p_j]$ is contained in the interior of the set $S_i \cup S_j$,
	and $\{k,\ell\} \in E_x \triangle E_y$ if and only if $[p_k,p_\ell]$ is contained in the interior of the set $S_k \cup S_\ell$.
	As shown in \ref{l:4squares1},
	the interiors of $S_i \cup S_j$ and $S_k \cup S_\ell$ are disjoint.
	Since $[p_i,p_j]$ and $[p_k,p_\ell]$ are not disjoint,
	it follows that $\{i,j\}, \{k,\ell\} \in E_x \cap E_y$.
	By \Cref{l:lineseg}\ref{l:lineseg2} and \Cref{l:lineseg}\ref{l:lineseg3}, 
	$p_{ij}$ (see \cref{eq:midpoint}) is the unique point in $[p_i,p_j]$ not contained in the interior of $S_i \cap S_j$.
	Similarly,
	$p_{k\ell}$ is the unique point in $[p_k,p_\ell]$ not contained in the interior of $S_k \cap S_\ell$.
	Since $[p_i,p_j]$ and $[p_k,p_\ell]$ intersect non-trivially but the interiors of $S_i \cup S_j$ and $S_k \cup S_\ell$ do not,
	we have $p_{ij}=p_{k\ell}$.
	Hence $\{i,j,k,\ell\}$ is a clique in $G$.
	By \Cref{l:4share},
	the squares $S_i,S_j,S_k,S_\ell$ share a corner.
	
	Suppose for contradiction that $\{i,k\} \in E_x \cap E_y$.
	Then $(i,j,k)$ is a cycle of length 3 in either $([n],E_x)$ or $([n],E_y)$,
	contradicting \Cref{mainlem0}.
	Hence $\{i,k\} \in E_x \triangle E_y$.
	By repeating the above argument we see that the edges $\{i,\ell\},\{j,k\},\{j,\ell\}$ also lie in $E_x \triangle E_y$,
	thus completing the proof.
\end{proof}

Interestingly, \Cref{l:4squares} also allows us to develop an easy-to-obtain upper bound for the maximum number of contacts possible in any homothetic square packing.
It is worth mentioning that the upper bound is not best possible;
indeed, the author believes that $4n - 6\sqrt{n} +2$ is the best-possible upper bound for the maximum number contacts for a homothetic packing of $n$ squares (the construction to obtain this upper bound is outlined in \Cref{sec:intro}).
As far as the author is aware, the upper bound given below is the lowest known upper bound.

\begin{proposition}\label{p:upperbound}
	Any homothetic packing of $n \geq 3$ squares has at most $4n-8$ contacts.
\end{proposition}

\begin{proof}
	Let $P$ be a homothetic square packing with contact graph $G=([n],E)$ and centres $p_1,\ldots,p_n$.
	Let $A$ be the set of cliques of size 4 contained in $G$.
	Choose any clique $K=\{i,j,k,\ell\} \in A$.
	By \Cref{l:4share},
	the squares $S_i,S_j,S_k,S_\ell$ all share a corner.
	Without loss of generality we may suppose $\{i,j\},\{k,\ell\}$ are the unique edges supported on $K$ that lie in $E_x \cap E_y$.
	Note that if we remove the edge $\{i,j\}$,
	the two cliques $\{i,k,\ell\}$ and $\{j,k,\ell\}$ will now form facial triangles in the straight-line embedding given by the centres of $P$.
	For each clique $K$ we now choose an edge $e_K \in E_x \cap E_y$ supported on $K$ and let $T_K,T'_K$ be the two corresponding facial triangles that result from the removal of $e_K$.
	Note that for two distinct cliques $K,K' \in A$,
	the triangles $T_K,T'_K,T_{K'},T'_{K'}$ must all be pairwise distinct by \Cref{l:cliques2share}.
	Define the subgraph $G' = ([n],E')$ by setting $E' := E \setminus \{e_K :K \in A\}$.
	By \Cref{l:4squares}\ref{l:4squares3},
	$G'$ is planar and the straight-line embedding given by the vertex map $i \mapsto p_i$ is a planar embedding.
	Furthermore,
	for every $K \in A$ the triangles $T_K$ and $T'_K$ are facial triangles of $G'$ with respect to the aforementioned straight-line embedding.
	It follows from Euler's formula that any embedding of a planar graph with $n$ vertices has at most $2n-4$ triangle faces.
	As every edge $e_K$ pairs up two triangles of the embedded graph $G'$ and no triangle is paired up more than once,
	we see that $|A| \leq n-2$.
	Since $G'$ is planar, it has at most $3n-6$ edges.
	Thus $|E| =|E'| + |A| \leq (3n-6) + (n-2) = 4n-8$ as required.
\end{proof}

Before we close the section, we first mention the following interesting result of Schramm.

\begin{theorem}[\cite{schramm93}]\label{t:schramm}
	Let $G=([n],E)$ be a planar graph with a planar embedding where every interior face is a triangle, the outer cycle of the embedded graph has length 4, and the only cycle of length at most 4 that contains vertices inside its interior is the outer cycle.
	Then $G$ is the contact graph of a homothetic square packing $P$.
\end{theorem}

Whilst being a very interesting result with a particularly beautiful proof -- for example, the proof involves observing a correspondence between square tilings and a concept known as an extremal metric -- we unfortunately cannot utilise \Cref{t:schramm} because of two important reasons.
Firstly,
\Cref{t:schramm} only applies to a very specific family of contact graphs:
planar graphs with $3n-7$ edges and no ``small'' cycles containing a vertex in their interior.
This assumption can be weakened to allow for any planar graphs with $3n-7$ edges, however the cost of doing so is that some squares could potentially have a radius of zero (something we are explicitly not allowing).
Secondly,
the union of the squares in the corresponding homothetic square packing will form a rectangle,
and so it is fairly easy to see that the radii cannot satisfy condition \ref{mainthm2} of \Cref{mainthm}.

\section{The weak generic condition}\label{sec:cycle}

We begin the section with the following definition.

\begin{definition}\label{def:wgc}
	The radii of a homothetic square packing are said to satisfy the \emph{weak generic condition} if they satisfy condition \ref{mainthm2} of \Cref{mainthm}.
\end{definition}

It should be noted that almost all choices of radii will satisfy the weak generic condition;
indeed it is sufficient that the radii form an algebraically independent set.
In this section we use the weak generic condition to determine certain properties about cycles in the coloured graphs $([n],E_x)$ or $([n],E_y)$.
We first require the following technical lemma.

\begin{lemma}\label{l:zigzag}
	Let $P$ be a homothetic packing of $n \geq 6$ squares.
	Suppose that the radii satisfy the weak generic condition.
	Further suppose that there exists a cycle $(n_1,\ldots,n_k)$ in the coloured subgraph $([n],E_x)$ such that for each $i \in [k]$, neither $x_{n_{i-1}} < x_{n_{i}} < x_{n_{i+1}}$ nor $x_{n_{i+1}} < x_{n_{i}} < x_{n_{i-1}}$ (here we set $x_{n_0} = x_{n_k}$ and $x_{n_{k+1}} = x_{n_1}$).
	Then $k = 4$ and the four squares $S_{n_1},S_{n_2},S_{n_3},S_{n_4}$ share a corner.
\end{lemma}

\begin{proof}	
	We first note that for the $x$-coordinates of the cycle $(n_1,\ldots,n_k)$ to have the required ``zigzagging'' property,
	$k$ must be even (and hence $k \geq 4$).
	By translating and reflecting $P$,
	we may assume that $x_{n_i} = (-1)^i r_{n_i}$ for each $i \in [k]$.
	By shifting and reversing the order of the cycle $(n_1,\ldots,n_k)$ as required,
	we may also suppose that  $y_{n_1} \leq y_{n_i}$ for all $i \in [k]$ and $y_{n_2} \leq y_{n_k}$.
	We note that our new ordering will imply $y_{n_1} < y_{n_i}$ for all odd $i \in [k] \setminus \{1\}$,
	and $y_{n_2} < y_{n_k}$,
	as otherwise the interiors of some of the squares in the packing will intersect.
	
	Choose any distinct $i,j \in [k]$.
	We now investigate two cases.
	In our first case, suppose that $i \equiv j \mod 2$ and $y_{n_i} \leq y_{n_j}$.
	So that $S_{n_i}$ and $S_{n_j}$ do not overlap, we must have $y_{n_i} <  y_{n_j}$. 
	As
	\begin{align*}
		|x_{n_i}-x_{n_j}| =|r_{n_i}-r_{n_j}| < r_{n_i} + r_{n_j}
	\end{align*}
	and the interiors of the squares $S_{n_i}$ and $S_{n_j}$ do not intersect, it follows that
	\begin{align}\label{opt0}
		y_{n_i} + r_{n_i} \leq y_{n_j} - r_{n_j}.
	\end{align}
	For our second case,
	suppose that $\{n_i,n_j\} \in E$ and $y_{n_i} - r_{n_i} \leq y_{n_j} - r_{n_j}$.
	As $\{n_i,n_{j}\} \in E_x$, we have $|y_{n_i} - y_{n_j}| \leq r_{n_i} + r_{n_j}$.
	The following two inequalities, however, cannot hold:
	\begin{align}
		y_{n_i} - r_{n_i} < y_{n_j} - r_{n_j} < y_{n_j} + r_{n_j} < y_{n_i} + r_{n_i} \label{eq:bad1}\\
		y_{n_i} - r_{n_i} = y_{n_j} - r_{n_j} < y_{n_i} + r_{n_i} = y_{n_j} + r_{n_j} \label{eq:bad2}
	\end{align}
	The inequality given by \cref{eq:bad1} cannot hold as, given $n_\ell$ is the single other vertex adjacent to $n_j$ in the cycle $(n_1,\ldots,n_k)$ (ignoring all other edges of the contact graph), the interiors of the two squares $S_{n_i}$ and $S_{n_\ell}$ will be forced to intersect so that both have an $x$-direction contact with $S_{n_j}$.
	The inequality given by \cref{eq:bad2} also cannot hold as it implies $r_{n_i}=r_{n_j}$,
	contradicting the weak generic condition (since $n \geq 6$).
	With this, we are left with 5 cases possible inequalites when $\{n_i,n_j\}$ is an edge in the cycle $(n_1,\ldots,n_k)$:
	\begin{enumerate}
		\item \label{opt1} $y_{n_i} - r_{n_i} < y_{n_i} + r_{n_i} = y_{n_j} - r_{n_j} < y_{n_j} + r_{n_j}$.
		\item \label{opt2} $y_{n_i} - r_{n_i} < y_{n_j} - r_{n_j} < y_{n_i} + r_{n_i} < y_{n_j} + r_{n_j}$.
		\item \label{opt3} $y_{n_i} - r_{n_i} < y_{n_j} - r_{n_j} < y_{n_i} + r_{n_i} = y_{n_j} + r_{n_j}$.
		\item \label{opt4} $y_{n_i} - r_{n_i} = y_{n_j} - r_{n_j} < y_{n_i} + r_{n_i} < y_{n_j} + r_{n_j}$.
		\item \label{opt5} $y_{n_i} - r_{n_i} = y_{n_j} - r_{n_j} < y_{n_j} + r_{n_j} < y_{n_i} + r_{n_i}$.
	\end{enumerate}

	We now use this analysis for the vertices $n_1,n_2,n_k$ to obtain some inequalities.
	By assumption we have $y_{n_1} \leq y_{n_2} < y_{n_k}$.
	First note that, as $k$ is even (and hence $k  \equiv 2 \mod 2$),
	$y_{n_2} + r_{n_2} \leq y_{n_k} - r_{n_k}$ by \cref{opt0}.
	Suppose that $y_{n_2} - r_{n_2} \leq y_{n_1} - r_{n_1}$.
	As $y_{n_1} \leq y_{n_2}$, it follows that $y_{n_2} - r_{n_2} \leq y_{n_2} - r_{n_1}$,
	which in turn implies $r_{n_1} < r_{n_2}$ (as $r_{n_1} \neq r_{n_2}$ by the weak generic condition).
	Hence $y_{n_1} + r_{n_1} < y_{n_2} + r_{n_2}$,
	and so the only possible case that can hold is case \ref{opt5} with $j=1$ and $i =2$,
	i.e.,
	\begin{align*}
		y_{n_1} - r_{n_1} = y_{n_2} - r_{n_2} < y_{n_1} + r_{n_1} < y_{n_2} + r_{n_2}.
	\end{align*}
	However, since $y_{n_2} + r_{n_2} \leq y_{n_k} - r_{n_k}$,
	this implies $y_{n_1} + r_{n_1} < y_{n_k} - r_{n_k}$,
	contradicting that $\{n_1,n_k\} \in E_x$.
	Hence $y_{n_1} - r_{n_1} < y_{n_2} - r_{n_2}$.
	Now suppose that $y_{n_k} - r_{n_k} \leq y_{n_1} - r_{n_1}$.
	As before,
	we see that the only possibility is for case \ref{opt5} to hold with $j=1$ and $i=k$,
	i.e.,
	\begin{align*}
		y_{n_1} - r_{n_1} = y_{n_k} - r_{n_k} < y_{n_1} + r_{n_1} < y_{n_k} + r_{n_k}.
	\end{align*}
	However, since $y_{n_2} + r_{n_2} \leq y_{n_k} - r_{n_k}$ and $y_{n_1} - r_{n_1} < y_{n_2} - r_{n_2}$,
	this implies $y_{n_2} + r_{n_2} < y_{n_2} - r_{n_2}$,
	a contradiction.
	Hence $y_{n_1} - r_{n_1} < y_{n_k} - r_{n_k}$.
	It follows that for each $\ell \in \{2,k\}$,
	one of cases \ref{opt1}, \ref{opt2}, or \ref{opt3} holds with $i =1$ and $j = \ell$.
	By observing the possible cases for $i=1,j=k$,
	we see that $y_{n_k} - r_{n_k} \leq  y_{n_1} + r_{n_1}$.
	From this, both cases \ref{opt1} and \ref{opt2} quickly run into a contradiction when $i=1$ and $j=2$.	
	Hence case \ref{opt3} holds for $i=1$ and $j=2$,
	i.e.,
	\begin{align}\label{eq:n2}
		y_{n_1} - r_{n_1} <y_{n_2} - r_{n_2} < y_{n_1} + r_{n_1} = y_{n_2} + r_{n_2}.
	\end{align}
	As $y_{n_2} + r_{n_2} \leq y_{n_k} - r_{n_k}$,
	the only possible case that can hold for $i=1$ and $j=k$ is case \ref{opt1},
	i.e.,
	\begin{align}\label{eq:nk}
		y_{n_1} - r_{n_1} <y_{n_1} + r_{n_1} = y_{n_k} - r_{n_k} < y_{n_k} + r_{n_k}.
	\end{align}
	
	Now we turn our attention to the vertex $n_3$.
	Since $1 \equiv 3 \mod 2$,
	we have $y_{n_1} < y_{n_3}$ and $y_{n_1} + r_{n_1} \leq y_{n_3} - r_{n_3}$ by \cref{opt0}.
	Hence $y_{n_2} + r_{n_2} \leq y_{n_3} - r_{n_3}$ and $y_{n_k} - r_{n_k} \leq y_{n_3} - r_{n_3}$ by \cref{eq:n2,eq:nk}.
	It follows that case \ref{opt1} holds for $i=2$ and $j=3$,
	and hence 
	\begin{align}\label{ineq}
		y_{n_1} - r_{n_1} < y_{n_2} - r_{n_2} < y_{n_1} + r_{n_1} = y_{n_2} + r_{n_2} = y_{n_k} - r_{n_k} = y_{n_3} - r_{n_3}.
	\end{align}
	From this we observe that the squares $S_1,S_2,S_3,S_k$ share a corner,
	with
	$\{n_1,n_2\},\{n_3,n_k\} \in E_x \setminus E_y$,
	$\{n_1,n_3\},\{n_2,n_k\} \in E_y \setminus E_x$,
	and $\{n_1,n_k\},\{n_2,n_3\} \in E_x \cap E_y$.
	
	Suppose for contradiction that $k >4$ (i.e., $k \geq 6$).
	Then,
	since $2 \equiv 4 \mod 2$ and $y_{n_2} \leq y_{n_4}$,
	we have
	$y_{n_2} +r_{n_2} \leq y_{n_4} - r_{n_4}$ by \cref{opt0}.
	Hence by \cref{ineq},
	$y_{n_k} - r_{n_k} \leq y_{n_4} - r_{n_4}$.
	If $y_{n_4} - r_{n_4} < y_{n_k} + r_{n_k}$ then $|y_{n_4} - y_{n_k}| < r_{n_4} + r_{n_k}$,
	which,
	when combined with $|x_{n_4} - x_{n_k}| < r_{n_4} + r_{n_k}$,
	contradicts that the interiors of $S_{n_4}$ and $S_{n_k}$ are disjoint.
	Hence $y_{n_k} + r_{n_k} \leq y_{n_4} - r_{n_4}$.
	Since $y_{n_k} - r_{n_k} < y_{n_k} + r_{n_k}$ and $y_{n_3} - r_{n_3} = y_{n_k} - r_{n_k}$ (\cref{ineq}),
	we have $y_{n_3} - r_{n_3} < y_{n_4} - r_{n_4}$.
	As one of cases \ref{opt1}, \ref{opt2} and \ref{opt3} must hold for $i=3$ and $j=4$,
	we have $y_{n_4} - r_{n_4} \leq y_{n_3} + r_{n_3}$.
	By combining this with the previous inequality of $y_{n_k} + r_{n_k} \leq y_{n_4} - r_{n_4}$,
	we have $y_{n_k} + r_{n_k} \leq y_{n_3} + r_{n_3}$.
	If $y_{n_k} + r_{n_k} = y_{n_3} + r_{n_3}$ then,
	since $y_{n_k} - r_{n_k} = y_{n_3} - r_{n_3}$ (\cref{ineq}),
	we would have $r_{n_3} = r_{n_4}$,
	contradicting the weak generic condition.
	Thus 
	\begin{equation}\label{eq:k3}
		y_{n_k} + r_{n_k} < y_{n_3} + r_{n_3}.
	\end{equation}
	Now observe the vertex $n_{k-1}$.
	If $y_{n_3} \leq y_{n_{k-1}}$ then by \cref{opt0,eq:k3} we have
	\begin{equation*}
		y_{n_k} + r_{n_k} < y_{n_3} + r_{n_3} \leq y_{n_{k-1}} - r_{n_{k-1}}
	\end{equation*}
	which implies $|y_{n_k} - y_{n_{k-1}}| > r_{n_k} + r_{n_{k-1}}$,
	contradicting that $\{n_{k-1},n_k\} \in E$.
	Thus $y_{n_{k-1}} < y_{n_{3}}$, and so, since $k$ is even and $(k-1) \equiv 3 \mod 2$,
	\begin{align}\label{eq:k-13}
		y_{n_{k-1}} + r_{n_{k-1}} \leq  y_{n_{3}} - r_{n_{3}}
	\end{align}
	by \cref{opt0}.
	By applying \cref{opt0} with $i=1$ and $j=k-1$ (since $k$ is even),
	and then applying the substitutions from \cref{ineq},
	we see that 
	\begin{align}\label{eq:k-1k}
		y_{n_k} - r_{n_k} \leq y_{n_{k-1}} - r_{n_{k-1}}.
	\end{align}
	However by combining \cref{ineq,eq:k-13,eq:k-1k} we see that
	\begin{align*}
		y_{n_{k-1}} + r_{n_{k-1}} \leq y_{n_3} - r_{n_3} = y_{n_k} - r_{n_k} \leq y_{n_{k-1}} - r_{n_{k-1}}
	\end{align*}
	contradicting that $r_{n_{k-1}} >0$.
	Hence $k=4$,
	completing the proof.
\end{proof}

Using the previous technical result,
we now prove that the weak generic condition forces all cycles in $([n],E_x)$ and $([n],E_y)$ to either be very long or generated by four squares sharing a corner.

\begin{lemma}\label{mainlem1}
	Let $P$ be a homothetic square packing with contact graph $G=([n],E)$, radii $r_1,\ldots,r_n$ and centres $p_1,\ldots, p_n$.
	Suppose that the radii satisfy the weak generic condition.
	If either of the graphs $([n],E_x), ([n],E_y)$ contains a cycle $(n_1,\ldots,n_k)$ with $k \leq n-2$,
	then $k=4$ and the squares $S_{n_1},S_{n_2},S_{n_3},S_{n_4}$ share a corner.
\end{lemma}

\begin{proof}
	Without loss of generality, we will assume $(n_1,\ldots,n_k)$ is a cycle of $([n],E_x)$.
	We will also fix that $n_{0} := n_k$ and $n_{k+1} := n_1$.	
	By \Cref{mainlem0},
	$k \geq 4$.
	Let $p_i = (x_i,y_i)$ for each $i \in [n]$.
	Define the function $\sigma: [n] \rightarrow \{-1,0,1\}$,
	where for each $i \in [n]$ we have:
	\begin{align*}
		\sigma_i :=
		\begin{cases}
			1 &\text{if } i = n_j \text{ for some $j \in [k]$ and } x_{n_{j-1}} < x_{n_{j}} < x_{n_{j+1}},\\
			-1 &\text{if } i = n_j \text{ for some $j \in [k]$ and } x_{n_{j+1}} < x_{n_{j}} < x_{n_{j-1}},\\
			0 &\text{otherwise.}
		\end{cases}
	\end{align*}
	Fix $s,t \in [k]$ to be distinct points where $x_{n_s} \leq x_{n_i} \leq x_{n_t}$ for each $i \in [k]$.
	By our choice of $\sigma$ we must have $\sigma_{n_s} = \sigma_{n_t} = 0$.
	As $k \leq n-2$,
	it follows that the map $\sigma$ has at least 4 zeroes.
	We observe the following property for any $i \in [k]$:
	\begin{align*}
		\frac{x_{n_{i+1}} - x_{n_i}}{|x_{n_{i+1}} - x_{n_i}|} + \frac{x_{n_{i}} - x_{n_{i-1}}}{|x_{n_i} - x_{n_{i-1}}|} = 2 \sigma_{n_i}.
	\end{align*}
	Adding this observation to the fact that $(n_1,\ldots,n_k)$ is a cycle of $([n],E_x)$, we have that
	\begin{align*}
		0 = \sum_{i=1}^k x_{n_{i+1}} - x_{n_i} = \sum_{i=1}^k \frac{x_{n_{i+1}} - x_{n_i}}{|x_{n_{i+1}} - x_{n_i}|} (r_{n_i} +r_{n_{i+1}}) = \sum_{i=1}^k 2 \sigma_{n_i} r_{n_i}  = 2\sum_{i=1}^n \sigma_i r_i.
	\end{align*}
	Hence $\sigma_i = 0$ for all $i \in [n]$,
	as the radii satisfy the weak generic condition.
	This implies that our cycle is ``zigzagging'',
	i.e., for each $i \in [k]$ we have that either $x_{n_{i-1}}< x_{n_i}$ and $x_{n_{i+1}}< x_{n_i}$,
	or $x_{n_{i-1}}> x_{n_i}$ and $x_{n_{i+1}}< x_{n_i}$.
	The result now follows from \Cref{l:zigzag}.
\end{proof}

Our next goal of this section is to prove the following:
if one of the coloured subgraphs contain a sufficiently long cycle,
then the weak generic condition will imply that the number of edges in the contact graph is bounded above by $2n-2$.
We first need the following technical lemma.

\begin{lemma}\label{l:forbiddensg}
	Let $P$ be a homothetic square packing with contact graph $G=([n],E)$, radii $r_1,\ldots,r_n$ and centres $p_1,\ldots, p_n$.
	If the radii of $P$ satisfy the weak generic condition,
	then $G$ does not contain the subgraph pictured in \Cref{fig:forbiddensg}.
\end{lemma}

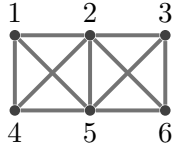
\begin{figure}[ht]
\begin{center}
\begin{tikzpicture}[scale=1]
		
		\node[vertex] (1) at (0,0) {};
		\node[vertex] (2) at (1,0) {};
		\node[vertex] (3) at (2,0) {};
		\node[vertex] (4) at (0,-1) {};
		\node[vertex] (5) at (1,-1) {};
		\node[vertex] (6) at (2,-1) {};
		
		\node (l1) at (0,0.3) {$1$};
		\node (l2) at (1,0.3) {$2$};
		\node (l3) at (2,0.3) {$3$};
		\node (l4) at (0,-1.3) {$4$};
		\node (l5) at (1,-1.3) {$5$};
		\node (l6) at (2,-1.3) {$6$};
			
		\draw[edge] (1)edge(2);
		\draw[edge] (1)edge(4);
		\draw[edge] (1)edge(5);
		\draw[edge] (2)edge(3);
		\draw[edge] (2)edge(4);
		\draw[edge] (2)edge(5);
		\draw[edge] (2)edge(6);
		\draw[edge] (3)edge(5);
		\draw[edge] (3)edge(6);
		\draw[edge] (4)edge(5);
		\draw[edge] (5)edge(6);
	\end{tikzpicture}
	\end{center}
	\caption{The forbidden subgraph of \Cref{l:forbiddensg} with vertices labelled 1 to 6.}
	\label{fig:forbiddensg}
\end{figure}

\begin{proof}Let $H$ be the graph pictured in \Cref{fig:forbiddensg}.
	Suppose for contradiction that $G$ contains a copy of $H$.
	Note that, as the radii $r_1,\ldots, r_n$ satisfy the weak generic condition,
	the radii $r_1,\ldots, r_6$ satisfy the weak generic condition.
	Hence without loss of generality we may assume that $G$ contains $H$ as a spanning subgraph (i.e., $n=6$).
	By relabelling the vertices of $G$ we may assume $H$ has the vertex labelling described in \Cref{fig:forbiddensg}.
	
	As $\{1,2,4,5\}$ (respectively, $\{2,3,5,6\}$) is a clique,
	by \Cref{l:4share} and \Cref{l:4squares}\ref{l:4squares3}
	there exist exactly two edges supported on $\{1,2,4,5\}$ (respectively, $\{2,3,5,6\}$) that are contained in $E_x \cap E_y$,
	and these two edges do not share any vertices.
	Suppose that $\{2,5\} \in E_x \cap E_y$.
	Then the edges $\{2,5\},\{1,4\},\{3,6\}$ are contained in $E_x \cap E_y$ and the rest of the edges of $H$ are contained in $E_x \triangle E_y$.
	It follows from \Cref{l:3squares,l:lineseg} that	the point $p_{25}$ described in \cref{eq:midpoint} is the unique point in the set $S_1 \cap S_2 \cap S_4 \cap S_5$ and also the unique point in the set $S_2 \cap S_3 \cap S_5 \cap S_6$.
	Hence $p_{25}$ is contained in every square of $P$.
	However this implies $G$ contains a clique of size 6,
	contradicting \Cref{l:4share}.
	Hence $\{2,5\}$ cannot be both an $x$- and $y$-direction contact.
	
	By relabelling the vertices of $H$,
	we can assume that the edges $\{1,5\},\{2,4\},\{2,6\},\{3,5\}$ are contained in $E_x \cap E_y$ and the rest of the edges of $H$ are contained in $E_x \triangle E_y$.
	By rotating $P$ we may assume that $\{2,5\} \in E_y$ (and hence $\{2,5\} \notin E_x$).
	Hence the edges of $H$ have the following colouring:
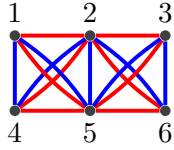
\begin{figure}[h]
\begin{center}
\begin{tikzpicture}[scale=1]
		
		\node[vertex] (1) at (0,0) {};
		\node[vertex] (2) at (1,0) {};
		\node[vertex] (3) at (2,0) {};
		\node[vertex] (4) at (0,-1) {};
		\node[vertex] (5) at (1,-1) {};
		\node[vertex] (6) at (2,-1) {};
		
		\node (l1) at (0,0.3) {$1$};
		\node (l2) at (1,0.3) {$2$};
		\node (l3) at (2,0.3) {$3$};
		\node (l4) at (0,-1.3) {$4$};
		\node (l5) at (1,-1.3) {$5$};
		\node (l6) at (2,-1.3) {$6$};
			
		\draw[edge, red] (1)edge(2);
		\draw[edge, blue] (1)edge(4);
		\draw[edge, red] (2)edge(3);
		\draw[edge, blue] (2)edge(5);
		\draw[edge, blue] (3)edge(6);
		\draw[edge, red] (4)edge(5);
		\draw[edge, red] (5)edge(6);
		
		\path[edge,red] (4) edge [bend right=10] node {} (2);
		\path[edge,blue] (4) edge [bend left=10] node {} (2);
		
		\path[edge,red] (2) edge [bend right=10] node {} (6);
		\path[edge,blue] (2) edge [bend left=10] node {} (6);
		
		\path[edge,red] (1) edge [bend right=10] node {} (5);
		\path[edge,blue] (1) edge [bend left=10] node {} (5);
		
		\path[edge,red] (5) edge [bend right=10] node {} (3);
		\path[edge,blue] (5) edge [bend left=10] node {} (3);
	\end{tikzpicture}
	\end{center}
	\label{fig:hgraph}
\end{figure}	
	\\
	Note that $(2,4,5,6)$ is a cycle in $([n],E_x)$.
	As the radii of $P$ satisfy the weak generic condition and $4 \leq n-2$,
	$\{2,4,5,6\}$ is a clique in $G$ and $\{4,6\} \in E$ by \Cref{mainlem1}.
	However the edge $\{4,6\}$ forms a cycle of length 3 in both $([n],E_x)$ and $([n],E_y)$, contradicting \Cref{mainlem0}.
\end{proof}

We are now ready to prove the final key lemma of the section.

\begin{lemma}\label{mainlem2}
	Let $P$ be a homothetic square packing with contact graph $G=([n],E)$, radii $r_1,\ldots,r_n$ and centres $p_1,\ldots, p_n$.
	Suppose that the radii satisfy the weak generic condition.
	If either of the subgraphs $([n],E_x), ([n],E_y)$ contains a cycle $(n_1,\ldots,n_k)$ with $k \geq n-1$,
	then $|E| \leq 2n - 2$.
\end{lemma}

\begin{proof}
	Without loss of generality, we will assume $(n_1,\ldots,n_k)$ is a cycle of $([n],E_x)$;
	this can be achieved by rotating $P$ by $90^\circ$ if necessary.
	If $n=4$ then $|E| \leq 2n-2$. 
	If $n=5$ with $|E| >2n-2$ then $G$ contains two cliques of size 4 sharing 3 vertices,
	which contradicts \Cref{l:cliques2share}.
	Hence we may assume $n \geq 6$.
	
	Suppose that $(n_1,\ldots,n_k)$ has a chord.
	By \Cref{mainlem0,mainlem1},
	any cycles of $([n],E_x)$ or $([n],E_y)$ have length 4, $n-1$ or $n$,
	and any cycle of length 4 is generated by four squares sharing a corner.
	So that the chord does not create a forbidden cycle in $([n],E_x)$,
	we must have $k=6$ and the chord must split the cycle into two cycles of length 4,
	each of which generated by four squares sharing a corner.
	However this implies that $G$ contains a copy of the graph pictured in \Cref{fig:forbiddensg},
	contradicting \Cref{l:forbiddensg}.	
	Hence the cycle $(n_1,\ldots,n_k)$ is chordless in $([n],E_x)$.
	It follows that the vertex set $\{n_1,\ldots,n_k\}$ cannot induce cycles of length 4 in either $([n],E_x)$ or $([n],E_y)$;
	any cycle of length 4 in the subgraph of $([n],E_y)$ induced by $\{n_1,\ldots,n_k\}$ necessitates a cycle of length 4 in the subgraph of $([n],E_x)$ induced by $\{n_1,\ldots,n_k\}$ (\Cref{mainlem1}),
	contradicting that $(n_1,\ldots,n_k)$ is a chordless cycle of $([n],E_x)$.
	In particular, no four squares in the set $\{n_1,\ldots,n_k\}$ can share a corner.
	
	Let $p_i = (x_i,y_i)$ for each $i \in [n]$.
	By relabelling the vertices (but maintaining that the order $(n_1,\ldots,n_k)$ forms a cycle),
	we will assume that $x_{n_1} \leq x_{n_i}$ for all $i \in [k]$.
	Fix $s \in [k]$ to be an index where $x_{n_s} \geq x_{n_i}$ for all $i \in [k]$.
	Define the function $\sigma: [n] \rightarrow \{-1,0,1\}$,
	where for each $i \in [n]$ we have:
	\begin{align*}
		\sigma_i :=
		\begin{cases}
			1 &\text{if } i = n_j \text{ for some $j \in [k]$ and } x_{n_{j-1}} < x_{n_{j}} < x_{n_{j+1}},\\
			-1 &\text{if } i = n_j \text{ for some $j \in [k]$ and } x_{n_{j+1}} < x_{n_{j}} < x_{n_{j-1}},\\
			0 &\text{otherwise}
		\end{cases}
	\end{align*}
	(here we set $n_0 = n_k$ and $n_{k+1} = n_1$).
	We now note four immediate properties of the map $\sigma$:
	(i)	$\sigma_{n_1} = \sigma_{n_s} = 0$;
	(ii) if $\sigma_{n_i} \neq 0$, then $\sigma_{n_{i+1}}$ is either 0 or equal to $\sigma_{n_i}$;
	(iii) if $\sigma_{n_i} = 1$ (respectively, $\sigma_{n_i} = -1$) and $\sigma_{n_{i+1}} = \ldots = \sigma_{n_{i+m}} = 0 \neq \sigma_{n_{i+m+1}}$ for some $m$,
	then $\sigma_{n_{i+m+1}} = 1$ (respectively, $\sigma_{n_{i+m+1}} = -1$) if $m$ is even and $\sigma_{n_{i+m+1}} = -1$ (respectively, $\sigma_{n_{i+m+1}} = 1$) if $m$ is odd;
	(iv) the map $\sigma$ is not constant (i.e., $\sigma_i = 0$ for all $i \in [n]$) as \Cref{l:zigzag} would then imply $k=4$, contradicting that $n \geq 6$.
	As $\sum_{i=1}^n \sigma_i r_i = 0$ (this follows from the same methods implemented in \Cref{mainlem1}),
	there exists some $a,b \in [k]$ such that $\sigma_{n_a} = 1$ and $\sigma_{n_b} = -1$,
	and $\sigma$ has at most 3 zeores (due to the radii satisfying the weak generic condition).
	It now follows that the number of indices $j \in [k]$ where $\sigma_{n_j}=0 $ must be non-zero and even;
	to see this, note that as we traverse the cycle $(n_1,\ldots,n_k)$ from $n_1$ back to $n_1$,
	the map $\sigma$ switches between $1$ and $-1$ (ignoring any zeroes inbetween) an even amount of times and only the $+ 1/-1$ switches generate odd-length strings of zeroes.
	Since $\sigma$ can have at most 3 zeroes,
	it follows that $\sigma$ has exactly two zeroes contained in the cycle $(n_1,\ldots,n_k)$.
	As $\sigma_{n_1} = \sigma_{n_s} =0$ and $x_{n_1} \leq x_{n_i}$ for all $i \in [k]$,
	we see that
	\begin{align*}
		\sigma_i =
		\begin{cases}
			1 &\text{if } i = n_j \text{ for some $1 < j < s$},\\
			-1 &\text{if } i = n_j \text{ for some $s< j \leq k$},\\
			0 &\text{otherwise.}
		\end{cases}
	\end{align*}
	Hence for $i \in [k]$ with $i \neq 1$ we have
	\begin{align*}
		x_{n_i} =
		\begin{cases}
			r_{n_i} + (x_{n_1} + r_{n_1}) + \sum_{j=2}^{i-1} 2r_{n_j} &\text{if } 1 < i \leq s,\\
			r_{n_i} + (x_{n_1} + r_{n_1}) + \sum_{j=i+1}^{k} 2r_{n_j} &\text{if } s \leq i \leq k.\\
		\end{cases}
	\end{align*}
	(Here we are using the convention that $\sum_{j=a}^b t_j =0$ if $a>b$.)
	We observe that if a pair $\{n_i,n_j\}$ with $1 \leq i <j \leq k$ is an edge of $([n],E_y)$, then either $j = i+1$ or $i \leq s \leq j$.
	Furthermore,
	the vertex $n_1$ is adjacent to a vertex $n_i$ in $([n],E)$ if and only if $i \in \{2,k\}$,
	and the vertex $n_s$ is adjacent to a vertex $n_j$ in $([n],E)$ if and only if $j \in \{s-1,s+1\}$.
	
	Fix $G'=(V',E')$ to be the subgraph of $G$ induced by $V':=\{n_1,\ldots,n_k\}$, and define $(V',E_x')$ and $(V',E'_y)$ to be the corresponding induced coloured subgraphs of $G'$.
	As $(n_1,\ldots,n_k)$ is a chordless cycle of $([n],E_x)$,
	the graph $(V',E_x')$ is a connected cycle with $k$ edges.
	Since any neighbours of $n_1$ and $n_s$ in $(V',E'_y)$ are also their neighbours in $(V',E'_x)$,
	the vertices $n_1,n_s$ are isolated vertices in $(V',E'_y \setminus E'_x)$,
	and hence any cycle in $(V',E'_y \setminus E'_x)$ contains at most $|V'|-2 \leq n-2$ vertices.
	As shown prior, $(V',E_y')$ does not contain any cycles of length at most $n-2$.
	Hence the subgraph $(V',E'_y \setminus E'_x)$ is a forest with at least 3 connected components (and thus at most $k-3$ edges) and isolated vertices $n_1,n_k$.
	It now follows that
	\begin{align*}
		|E'| = |E_x'| + |E_y' \setminus E_x'| \leq  k + (k-3) = 2k - 3.
	\end{align*}
	Hence if $k=n$,
	then $|E|= |E'| \leq 2n-3$ and we are done.
	
	Now suppose instead that $k = n-1$.
	Further suppose that the vertex $n$ is adjacent to more than 2 vertices in $([n],E_x)$.
	Then the only possible structure that $([n],E_x)$ can take without generating a cycle of length 3 (contradicting \Cref{mainlem0}), or generating a cycle of length more than 4 and less than $n-1$ is the following graph with $n=7$ and vertex $n$ as the centre vertex:
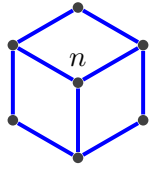
\begin{figure}[h]
\begin{center}
\begin{tikzpicture}[scale=1]
		
		\node[vertex] (1) at (0,1) {};
		\node[vertex] (2) at (0.866,0.5) {};
		\node[vertex] (3) at (0.866,-0.5) {};
		\node[vertex] (4) at (0,-1) {};
		\node[vertex] (5) at (-0.866,-0.5) {};
		\node[vertex] (6) at (-0.866,0.5) {};
		\node[vertex] (n) at (0,0) {};
		\node (nname) at (0,0.3) {$n$};
			
		\draw[edge, blue] (1)edge(2);
		\draw[edge, blue] (2)edge(3);
		\draw[edge, blue] (3)edge(4);
		\draw[edge, blue] (4)edge(5);
		\draw[edge, blue] (5)edge(6);
		\draw[edge, blue] (6)edge(1);
		
		\draw[edge, blue] (2)edge(n);
		\draw[edge, blue] (4)edge(n);
		\draw[edge, blue] (6)edge(n);
	\end{tikzpicture}
	\end{center}
	\label{fig:badsub}
\end{figure}	
	\\
	As each cycle of length 4 must generate a clique of size 4 in $G$,
	$G$ must contain a copy of the graph featured in \Cref{fig:forbiddensg},
	contradicting \Cref{l:forbiddensg}.
	Hence we may suppose that $n$ is adjacent to at most two vertices in $([n],E_x)$.
	Furthermore,
	if $n$ is adjacent to vertices $n_i,n_j \in V'$ in $([n],E_x)$ with $i <j$,
	then an analysis of the possible cycles in $([n],E_x)$ show that $n,n_i,n_j$ must be contained in a cycle of 4.
	As $n_1,n_s$ are isolated vertices in the forest $(V',E_y' \setminus E_x')$,
	any cycle of $([n],E_y \setminus E_x)$ has length at most $n-2$ and does not contain $n_1,n_s$.
	Hence by \Cref{mainlem1},
	any cycle of $([n],E_y \setminus E_x)$ has length 4 and is induced by 4 squares.
	However any cycle in $([n],E_y)$ induced by 4 squares sharing a corner will not be a cycle in $([n],E_y \setminus E_x)$, since exactly two edges of the cycle must also be contained in $E_x$.
	Hence $([n],E_y \setminus E_x)$ is a forest.
	
	Without loss of generality, one of three possible cases must now hold:
	\begin{enumerate}
		\item \label{case1} $n$ is adjacent to neither $n_1$ nor $n_s$ in $([n],E_y \setminus E_x)$,
		\item \label{case2} $n$ is adjacent to $n_1$ but not $n_s$ in $([n],E_y \setminus E_x)$,
	or 
	\item \label{case3} $n$ is adjacent to both $n_1$ and $n_s$ in $([n],E_y \setminus E_x)$.
	\end{enumerate}
	First suppose case \ref{case1} holds.
	Since $n$ is adjacent to at most 2 vertices in $([n],E_x)$ and $([n],E_y \setminus E_x)$ is a forest with at least $c_y$ connected components,
	we see that
	\begin{align*}
		|E| = |E_x| + |E_y \setminus E_x| \leq (n-1+2) + (n-c_y) = 2n-c_y +1 \leq 2n-2
	\end{align*}
	and we are done.
	
	Now suppose case \ref{case2} holds.
	By a similar counting method to that above,
	we observe that either $G$ has at most $2n-2$ edges, 
	or $([n],E_y \setminus E_x)$ has exactly 2 connected components (with $n_s$ as an isolated vertex) and $n$ is adjacent to exactly 2 vertices in $([n],E_x)$.
	Suppose for contradiction that the latter holds.
	Let $(a,b,c,n)$ be the cycle of length 4 in $([n],E_x)$ that contains $n$ and its two neighbours $a,c$.
	Since the squares $S_a,S_b,S_c,S_n$ share a corner (\Cref{mainlem1}) and the cycle is ordered $(a,b,c,n)$,
	we must have $\sigma_b = 0$.
	As $n$ is not adjacent to $n_s$ and $(a,b,c)$ is a path in the cycle $(n_1,\ldots,n_k)$,
	we have $a=n_2$, $b =n_1$, $c=n_k$ and $x_n+r_n = x_{n_1} + r_{n_1}$.
	An analysis of the $x$-coordinates of the various centres of $P$ show that $n$ is only adjacent to $n_1,n_2,n_k$ in $G$.
	Since $\{n,n_2\},\{n,n_k\} \in E_x$, the forest $([n],E_y \setminus E_x)$ has at least 3 connected components, contradicting our earlier assumption.
	
	Finally,
	suppose case \ref{case3} holds.
	As $\{n_1,n\},\{n_s,n\} \in E_y \setminus E_x$,
	we have $|x_{n_1} - x_{n}| < r_{n_1} + r_{n}$ and $|x_{n_s} - x_{n}| < r_{n_s} + r_{n}$.
	Using our prior knowledge of the positions of the $x$-coordinates for vertices in $V'$,
	it follows that for each $i \in [k] \setminus \{1,s\}$ we have
	\begin{align*}
		x_n - r_n - r_{n_i} < x_n - r_n < x_{n_1} + r_{n_1} < x_{n_i} < x_{n_s} - r_{n_s} < x_n +r_n < x_n + r_n + r_{n_i},
	\end{align*}
	and so $|x_n-x_{n_i}| < r_n + r_{n_i}$.
	Hence $n$ has no neighbours in $([n],E_x)$ and $|E_x| = n-1$.
	As $([n],E_y \setminus E_x)$ is a forest,
	we now see that
	\begin{align*}
		|E| = |E_x| + |E_y \setminus E_x| \leq (n-1) + (n-1) = 2n-2.
	\end{align*}
	This completes the proof.
\end{proof}

\section{Proof of \texorpdfstring{\Cref{mainthm}}{Theorem 1.1}}\label{sec:main}

Before we prove \Cref{mainthm},
we first require the following two technical lemmas.

\begin{lemma}\label{l:suff}
	Suppose that for a given set of radii $r_1,\ldots,r_n$ with $n \geq 2$,
	there exists a homothetic square packing with radii $r_1,\ldots,r_n$ and $k$ contacts.
	Then for any choice of $r_{n+1} >0$,
	there exists a homothetic square packing with radii $r_1,\ldots,r_{n+1}$ and at least $k+2$ contacts.
\end{lemma}

\begin{proof}
	Let $P= \{S_1,\ldots,S_n\}$ be a homothetic square packing with contact graph $G=([n],E)$, radii $r_1,\ldots,r_n$ and centres $p_1,\ldots, p_n$,
	where $|E| = k$.
	We may assume that $G$ is connected;
	indeed if it was not, we could translate one connected component of $P$ until it was in contact with another and increase the amount of contacts by at least 1.
	Define the closed set
	\begin{align*}
		X := \left\{ z \in \mathbb{R}^2 : (r_{n+1} S^\circ + z) \cap S_i = \emptyset \text{ for all } i \in [n] \right\}.
	\end{align*}
	Note that for any point $z \in X$,
	the interior of the set $r_{n+1} S + z$ will not intersect $\bigcup_{i=1}^n S_i$,
	and the set $r_{n+1} S + z$ will intersect $\bigcup_{i=1}^n S_i$ if and only if $z \in \partial X$.
	Hence for any $z \in \partial X$,
	the set $\{S_1,\ldots,S_n , r_{n+1} S + z\}$ will be a homothetic square packing with at least $k+1$ contacts.
	It follows that we now need only find a point $z \in \partial X$  such that $r_{n+1} S + z$ is in contact with at least two squares in $P$.
	
	Choose any point $z' \in \partial X$.
	If the square $r_{n+1} S +z'$ is in contact with two or more squares in $P$ then we are done.
	Suppose that $r_{n+1} S +z'$ is in contact with exactly one square $S_i$.
	Given $\partial S$ is the boundary of the standard square,
	$z'$ is an element of the set $C := (r_i + r_{n+1})\partial S + p_i$.
	It is immediate that $C \cap X^\circ = \emptyset$.
	As the sets $C$ and $\partial X$ are closed,
	the set $C \cap \partial X$ is a non-empty closed subset of $C$.
	The boundary of $C \cap \partial X$ with respect to the ambient space $C$ exists as $C \not\subset \partial X$;
	indeed if $C \subset \partial X$,
	then $S_i$ would not be in contact with any other square in $P$,
	contradicting that $G$ is connected and $n \geq 2$.
	Choose a point $z \in C$.
	If $z$ is not contained in $C \setminus \partial X$ then, since $C \cap X^\circ = \emptyset$, $z \notin X$ and so the set $r_{n+1} S + z$ will be in contact with $S_i$ and intersect the interior of another square $S_j \neq S_i$.
	If $z$ is in the interior of $C \cap \partial X$ with respect to the ambient space $C$,
	then $r_{n+1} S + z$ is in contact with $S_i$ and but it will not intersect any other square in $P$.
	Hence if we choose a point $z$ in the boundary of $C \cap \partial X$ with respect to the ambient space $C$,
	then $r_{n+1} S + z$ will be in contact with at least two squares in $P$.
\end{proof}

\begin{lemma}\label{l:ygraph}
	Let $P$ be a homothetic square packing with contact graph $G=([n],E)$, radii $r_1,\ldots,r_n$ and centres $p_1,\ldots, p_n$.
	Suppose that for some $1 \leq s \leq n-1$,
	the following holds:
	\begin{enumerate}
		\item $\sum_{i=1}^s r_i = \sum_{i=s+1}^n r_i$,
		\item $p_1 = (r_1,r_1)$ and $p_{s+1} = (r_{s+1},-r_{s+1})$, and
		\item $p_i = (r_i + \sum_{j=1}^{i-1} 2r_j, r_i)$ for all $2 \leq i \leq s$ and $p_i = (r_i + \sum_{j=s+1}^n 2r_j, -r_i)$ for all $s+2 \leq i \leq n$.
	\end{enumerate}
	Then the graph $([n],E_y)$ is connected.
	(See \Cref{fig:2lines} for an example of such a packing.)
\end{lemma}

\begin{proof}
	Fix $p_i = (x_i,y_i)$ for each $i \in [n]$.
	Define for each $i \in [n]$ the closed interval
	\begin{align*}
		I_i := [x_i- r_i ,x_i +  r_i].
	\end{align*}
	Since $\bigcup_{i=1}^s I_i = \bigcup_{i=s+1}^n I_i$,
	we observe that every interval $I_i$ for $i \leq s$ must intersect at least one set $I_j$ for $j \geq s+1$.
	Choose any $i \in [s-1]$ and let $j \in [n]$ be the largest index such that $I_i \cap I_j \neq \emptyset$;
	by our previous observation we note that $j \geq s+1$.
	Suppose that $x_i +r_i \geq x_j + r_j$.
	If $j =n$, then $i = s$,
	contradicting that $i \in [s-1]$.
	If $j <n$ then $I_i \cap I_{j+1} \neq \emptyset$ as $x_i +r_i \geq x_j + r_j = x_{j+1} - r_{j+1}$,
	contradicting the maximality of $j$.
	Hence $x_i +r_i < x_j + r_j$.
	Since $x_{i+1} - r_{i+1} = x_i +r_i$,
	it follows that $I_j \cap I_{i+1} \neq \emptyset$.
	From this we can deduce that for each $i \in [n]$ where $i \leq s-1$,
	there exists $j \in [n]$ such that $j \geq s+1$ and $\{i,j\},\{i+1,j\} \in E_y$.
	By a similar technique we can show that for each $j \in [n]$ where $s+1 \leq j \leq n-1$,
	there exists $i \in [n]$ where $i \leq s$ such that $\{i,j\},\{i,j+1\} \in E_y$.
	With this we can construct two paths $P_1,P_2 \in ([n],E_y)$ such that $P_1$ contains every vertex $1 \leq i \leq s$ and at least one vertex $j \geq s+1$,
	and $P_2$ contains every vertex $s+1 \leq j \leq n$ and at least one vertex $i \leq s$.
	Hence the graph $([n],E_y)$ is connected.	
\end{proof}

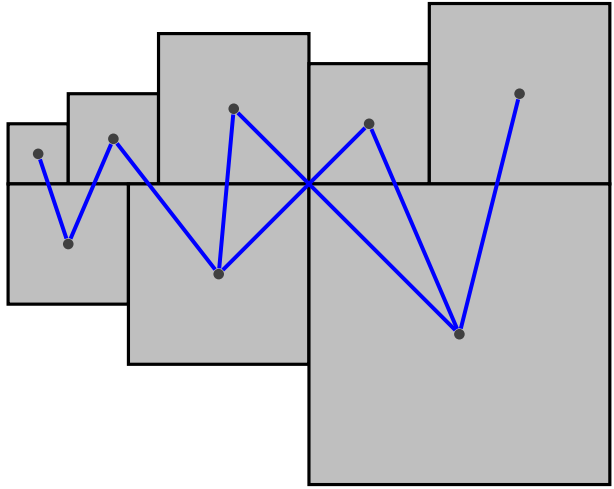
\begin{figure}[ht]
\begin{center}
\begin{tikzpicture}[scale=0.4]
		
		\begin{scope}[shift={(1,1)},scale=1,rotate=0]
		    \draw[very thick, fill=lightgray] (-1,-1) rectangle (1,1);
			\node[vertex] (1) at (0,0) {};
		\end{scope}
			
		\begin{scope}[shift={(3.5,1.5)},scale=1.5,rotate=0]
		    \draw[very thick, fill=lightgray] (-1,-1) rectangle (1,1);
		    \node[vertex] (2) at (0,0) {};
		\end{scope}
			
		\begin{scope}[shift={(7.5,2.5)},scale=2.5,rotate=0]
		    \draw[very thick, fill=lightgray] (-1,-1) rectangle (1,1);
		    \node[vertex] (3) at (0,0) {};
		\end{scope}
			
		\begin{scope}[shift={(12,2)},scale=2,rotate=0]
		    \draw[very thick, fill=lightgray] (-1,-1) rectangle (1,1);
		    \node[vertex] (4) at (0,0) {};
		\end{scope}
			
		\begin{scope}[shift={(17,3)},scale=3,rotate=0]
		    \draw[very thick, fill=lightgray] (-1,-1) rectangle (1,1);
		    \node[vertex] (5) at (0,0) {};
		\end{scope}
			
		\begin{scope}[shift={(2,-2)},scale=2,rotate=0]
		    \draw[very thick, fill=lightgray] (-1,-1) rectangle (1,1);
		    \node[vertex] (6) at (0,0) {};
		\end{scope}
			
		\begin{scope}[shift={(7,-3)},scale=3,rotate=0]
		    \draw[very thick, fill=lightgray] (-1,-1) rectangle (1,1);
		    \node[vertex] (7) at (0,0) {};
		\end{scope}
			
		\begin{scope}[shift={(15,-5)},scale=5,rotate=0]
		    \draw[very thick, fill=lightgray] (-1,-1) rectangle (1,1);
		    \node[vertex] (8) at (0,0) {};
		\end{scope}

		\draw[edge, blue] (1)edge(6);
		\draw[edge, blue] (2)edge(6);
		\draw[edge, blue] (2)edge(7);
		\draw[edge, blue] (3)edge(7);
		\draw[edge, blue] (3)edge(8);
		\draw[edge, blue] (4)edge(7);
		\draw[edge, blue] (4)edge(8);
		\draw[edge, blue] (5)edge(8);
		
		
	\end{tikzpicture}
	\end{center}
	\caption{An example of the construction from \Cref{l:ygraph}. The radii of the squares on the top row from left to right are $r_1 = 1$, $r_2 = 1.5$, $r_3 = 2.5$, $r_4 = 2$ and $r_5 = 3$, and the radii of the squares on the bottom row from left to right are $r_6=2$, $r_7 = 3$, $r_8 = 5$.
	As can be seen, the graph $([n],E_y)$ for this homothetic square packing is connected.}
	\label{fig:2lines}
\end{figure}

With this we are finally ready to prove \Cref{mainthm}.

\begin{proof}[Proof of \Cref{mainthm}]
	Suppose that \ref{mainthm2} holds,
	i.e., the radii satisfy the weak generic condition.
	Let $P$ be a homothetic square packing with contact graph $G=([n],E)$, radii $r_1,\ldots,r_n$ and centres $p_1,\ldots,p_n$.
	If either $([n],E_x)$ or $([n],E_y)$ contains a cycle with at least $n-1$ vertices,
	then $|E| \leq 2n-2$ by \Cref{mainlem2}.
	Suppose that any cycle in either $([n],E_x)$ or $([n],E_y)$ has length at most $n-2$.
	By \Cref{mainlem0,mainlem1},
	every cycle in $([n],E_x)$ (respectively, $([n],E_y)$) has length 4 and is generated by 4 squares sharing a corner.
	Let $C_1^x,\ldots,C_k^x \subset E_x$ and $C_1^y,\ldots,C_k^y \subset E_y$ be the cycles of $([n],E_x)$ and $([n],E_y)$ labelled so that for each $1 \leq i \leq k$ we have $C_i^x \cap C_i^y = \{ e_i, f_i\}$ for some edges $e_i,f_i$;
	this corresponds to the cycles $C_i^x$ and $C_i^y$ being generated by the same 4 squares sharing a corner.
	Define $E_x' := E_x \setminus \{e_1,\ldots,e_k\}$ and $E_y' := E_y \setminus \{f_1,\ldots,f_k\}$.
	It is immediate that $E =  E'_x \cup E'_y$ and both $([n],E_x')$ and $([n],E_y')$ are trees,
	and so
	\begin{align*}
		|E| \leq |E'_x| + |E_y'| \leq (n-1) + (n-1) = 2n - 2.
	\end{align*}
	Hence \ref{mainthm1} holds.
	
	Now suppose \ref{mainthm2} does not hold;
	i.e., there exists a map $\sigma : [n] \rightarrow \{-1,0,1\}$ with $\sum_{i=1}^n \sigma_i r_i = 0$ and $\sigma_{n_0} \neq 0$, $\sigma_{n_1} = \ldots = \sigma_{n_4} = 0$ for distinct vertices $n_0,\ldots,n_4 \in [n]$.
    By reordering the indices we may assume that
    $\sigma_1,\ldots,\sigma_s = 1$, $\sigma_{s+1},\ldots,\sigma_{n-t} = -1$ and $\sigma_{n-t+1},\ldots,\sigma_n = 0$ for some $s \geq 1$ and $t \geq 4$.
    With this reordering we have $\sum_{i=1}^s r_i = \sum_{i=s+1}^{n-t} r_i$.
    By \Cref{l:suff},
    it is sufficient to consider the case where $t = 4$;
    if there exists homothetic square packing with at least $2(n-t+4) - 1$ contacts using only the first $n-t+4$ radii,
    then there exists a homothetic square packing with at least $2n-1$ contacts using all the radii.
    
	For each $1 \leq i \leq n$, set $p_i = (x_i,y_i)$,
	where
	\begin{align*}
		x_i :=
		\begin{cases}
			r_1 &\text{if } i=1, \\
			r_i + \sum_{j=1}^{i-1} 2 r_j &\text{if } 2 \leq i \leq s, \\
			r_{s+1} &\text{if } i=s+1, \\
			r_i + \sum_{j=s+1}^{i-1} 2 r_j &\text{if } s+2 \leq i \leq n-4, \\
			-r_i &\text{if } i = n-3 \text{ or } i = n-2, \\
			r_i + \sum_{j=1}^{s} 2 r_j &\text{if } i = n-1 \text{ or } i = n,
		\end{cases}
	\end{align*}
	\begin{align*}
		y_i :=
		\begin{cases}
			r_i &\text{if } 1 \leq i \leq s \text{ or } i = n-3\text{ or } i = n-1, \\
			-r_i &\text{if } s+1 \leq i \leq n-4 \text{ or } i = n-2\text{ or } i = n.
		\end{cases}
	\end{align*}
    The family $P = \{ r_i S + p_i : i \in [n]\}$ now defines a homothetic square packing with contact graph $G = ([n],E)$;
    see \Cref{fig:construction} for an example of the construction.
    First note that $E_x \setminus E_y$ contains the paths $(n-3, 1, \ldots, s, n-1)$ and $(n-2, s+1, \ldots, n-4, n)$,
    hence $|E_x \setminus E_y| \geq n-2$.
    By \Cref{l:ygraph},
    the graph $([n],E_y)$ restricted to the vertices $1,\ldots,n-4$ is connected,
    and hence has at least $n-5$ edges.
    There are also 6 extra edges in $([n],E_y)$;
    $\{1, n-2\}, \{s+1, n-3\},\{ s, n\}, \{n-4, n-1\} , \{n-2,n-3\}, \{n-1,n\}$.
    Hence $|E_y| = n-5 + 6 = n+1$.
    It now follows that $|E| \geq 2n -1$,
    and so \ref{mainthm1} does not hold. 
\end{proof}

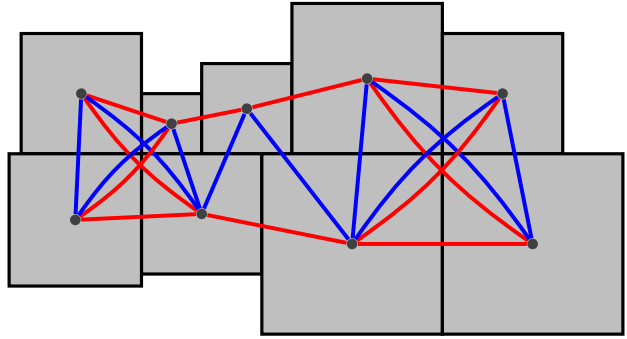
\begin{figure}[ht]
\begin{center}
\begin{tikzpicture}[scale=0.4]
		
		\begin{scope}[shift={(1,1)},scale=1,rotate=0]
		    \draw[very thick, fill=lightgray] (-1,-1) rectangle (1,1);
			\node[vertex] (1) at (0,0) {};
		\end{scope}
			
		\begin{scope}[shift={(3.5,1.5)},scale=1.5,rotate=0]
		    \draw[very thick, fill=lightgray] (-1,-1) rectangle (1,1);
		    \node[vertex] (2) at (0,0) {};
		\end{scope}
			
		\begin{scope}[shift={(7.5,2.5)},scale=2.5,rotate=0]
		    \draw[very thick, fill=lightgray] (-1,-1) rectangle (1,1);
		    \node[vertex] (3) at (0,0) {};
		\end{scope}

		\begin{scope}[shift={(2,-2)},scale=2,rotate=0]
		    \draw[very thick, fill=lightgray] (-1,-1) rectangle (1,1);
		    \node[vertex] (4) at (0,0) {};
		\end{scope}
			
		\begin{scope}[shift={(7,-3)},scale=3,rotate=0]
		    \draw[very thick, fill=lightgray] (-1,-1) rectangle (1,1);
		    \node[vertex] (5) at (0,0) {};
		\end{scope}

		\begin{scope}[shift={(-2,2)},scale=2,rotate=0]
		    \draw[very thick, fill=lightgray] (-1,-1) rectangle (1,1);
		    \node[vertex] (6) at (0,0) {};
		\end{scope}
			
		\begin{scope}[shift={(-2.2,-2.2)},scale=2.2,rotate=0]
		    \draw[very thick, fill=lightgray] (-1,-1) rectangle (1,1);
		    \node[vertex] (7) at (0,0) {};
		\end{scope}

		\begin{scope}[shift={(12,2)},scale=2,rotate=0]
		    \draw[very thick, fill=lightgray] (-1,-1) rectangle (1,1);
		    \node[vertex] (8) at (0,0) {};
		\end{scope}
			
		\begin{scope}[shift={(13,-3)},scale=3,rotate=0]
		    \draw[very thick, fill=lightgray] (-1,-1) rectangle (1,1);
		    \node[vertex] (9) at (0,0) {};
		\end{scope}

		\draw[edge, blue] (1)edge(4);
		\draw[edge, blue] (2)edge(4);
		\draw[edge, blue] (2)edge(5);
		\draw[edge, blue] (3)edge(5);
		\draw[edge, blue] (6)edge(7);
		\draw[edge, blue] (8)edge(9);

		\draw[edge, red] (6)edge(1);
		\draw[edge, red] (1)edge(2);
		\draw[edge, red] (2)edge(3);
		\draw[edge, red] (3)edge(8);		
		\draw[edge, red] (7)edge(4);
		\draw[edge, red] (4)edge(5);
		\draw[edge, red] (5)edge(9);
		
		\path[edge,red] (6) edge [bend right=10] node {} (4);
		\path[edge,blue] (6) edge [bend left=10] node {} (4);
		
		\path[edge,red] (7) edge [bend right=10] node {} (1);
		\path[edge,blue] (7) edge [bend left=10] node {} (1);
		
		\path[edge,red] (3) edge [bend right=10] node {} (9);
		\path[edge,blue] (3) edge [bend left=10] node {} (9);
		
		\path[edge,red] (5) edge [bend right=10] node {} (8);
		\path[edge,blue] (5) edge [bend left=10] node {} (8);
	\end{tikzpicture}
	\end{center}
	\caption{An example of the construction from \Cref{mainthm} with $17 > 2\cdot 9 - 2$ contacts. The construction is possible because the sum of the radii for the middle three squares on the top row (from left to right; 1, 1.5, 2.5) is equal to the sum of the radii for the middle two squares on the bottom row (from left to right; 2, 3), and the radii sum equality does not use the radii of at least four squares. The four squares that are not involved in the summation are then placed on the left and right to form the two cliques of size 4.}
	\label{fig:construction}
\end{figure}

\section{Square packings allowing rotations and face-to-face contacts}\label{sec:facetoface}

Throughout the paper we have only been interested in homothetic square packings.
In this section we will relax the condition that each square is a homothetic copy of $S$.
A \emph{similar copy of $S$} is a set $r R_\theta S + p$,
where $r>0$, $p \in \mathbb{R}^2$, $\theta \in [0,\pi/2)$ and $R_\theta$ is the $2 \times 2$ matrix representing anticlockwise rotation of the plane by $\theta$ radians.
With this, we define a \emph{packing of $n$ squares}, or \emph{square packing} for short, to be a set $P = \{S_1,\ldots,S_n\}$ of similar copies of $S$ with pairwise disjoint interiors.
The \emph{centres}, \emph{radii} and \emph{angles} of a square packing $P = \{S_1,\ldots,S_n\}$ with $S_i = r_i R_{\theta_i}S +p_i$ for each $i \in [n]$ will be the vectors $p_1,\ldots,p_n$, positive scalars $r_1,\ldots,r_n$ and the angles $\theta_1,\ldots, \theta_n$ respectively.

Interestingly, the amount of contacts of a packing of $n$ squares is not bounded by $2n-2$ even when the radii do not satisfy any polynomial equation with rational coefficients.

\begin{proposition}\label{p:sim}
	For each $n \geq 5$,
	there exists a packing of $n$ squares with algebraically independent radii and more than $2n-2$ contacts.
\end{proposition}

\begin{proof}
	Fix $P = \{S_1,\ldots,S_n\}$ to be the square packing with radii $r$, centres $p$ and angles $\theta$ defined as follows:
	\begin{enumerate}
		\item $r_1 := 1/3$, $r_3 := 2/3$, $r_n :=  2\sqrt{2}/3$ and $r_i:=1$ otherwise.
		\item $\theta_n := \pi/4$ and $\theta_i := 0$ otherwise.
		\item $p_1 := (-1/3,1/3)$, $p_2 := (-1,-1)$, $p_3 := (2/3, -2/3)$, $p_4 := (1,1)$,
		for each $5 \leq i \leq n-1$ we have 
		\begin{align*}
			p_i :=
			\begin{cases}
			    (i - 4 + 4/3 , -1) &\text{if $i$ is odd},\\
			    (i-3 , 1) &\text{if $i$ is even},
			\end{cases}
		\end{align*}				
		and $p_n = (-4/3,4/3)$.
	\end{enumerate}
	See \Cref{fig:1sqrotated} (left) for the described square packing $P$ with $n = 7$.
	We note that $P$ has $2n-1$ contacts. 
	Furthermore,
	for small perturbations of the vector $r$ where $r_1$ is decreased,
	we can always form a square packing similar to that indicated in \Cref{fig:1sqrotated} (right),
	which will also always have $2n-1$ contacts.
	Hence there exists a packing of $n$ squares with algebraically independent radii and $2n-1$ contacts.
\end{proof}

\begin{figure}[ht]
\begin{center}
\begin{tikzpicture}[scale=1]
		\begin{scope}[shift={(-1/3,1/3)},scale=1/3,rotate=0]
		    \draw[very thick, fill=lightgray] (-1,-1) rectangle (1,1);
		\end{scope}
		
		\begin{scope}[shift={(-1,-1)},scale=1,rotate=0]
		    \draw[very thick, fill=lightgray] (-1,-1) rectangle (1,1);
		\end{scope}
		
		\begin{scope}[shift={(2/3,-2/3)},scale=2/3,rotate=0]
		    \draw[very thick, fill=lightgray] (-1,-1) rectangle (1,1);
		\end{scope}
		
		\begin{scope}[shift={(1,1)},scale=1,rotate=0]
		    \draw[very thick, fill=lightgray] (-1,-1) rectangle (1,1);
		\end{scope}
		
		\begin{scope}[shift={(7/3,-1)},scale=1,rotate=0]
		    \draw[very thick, fill=lightgray] (-1,-1) rectangle (1,1);
		\end{scope}
		
		\begin{scope}[shift={(3,1)},scale=1,rotate=0]
		    \draw[very thick, fill=lightgray] (-1,-1) rectangle (1,1);
		\end{scope}
		
		\begin{scope}[shift={(-4/3,4/3)},scale=0.943,rotate=45]
		    \draw[very thick, fill=lightgray] (-1,-1) rectangle (1,1);
		\end{scope}

		
		
	\end{tikzpicture}\qquad\qquad
    \begin{tikzpicture}[scale=1]
		\begin{scope}[shift={(-1/4,1/4)},scale=1/4,rotate=0]
		    \draw[very thick, fill=lightgray] (-1,-1) rectangle (1,1);
		\end{scope}
		
		\begin{scope}[shift={(-1.2,-1.2)},scale=1.2,rotate=0]
		    \draw[very thick, fill=lightgray] (-1,-1) rectangle (1,1);
		\end{scope}
		
		\begin{scope}[shift={(3/5,-3/5)},scale=3/5,rotate=0]
		    \draw[very thick, fill=lightgray] (-1,-1) rectangle (1,1);
		\end{scope}
		
		\begin{scope}[shift={(0.8,0.8)},scale=0.8,rotate=0]
		    \draw[very thick, fill=lightgray] (-1,-1) rectangle (1,1);
		\end{scope}
		
		\begin{scope}[shift={(2,-0.8)},scale=0.8,rotate=0]
		    \draw[very thick, fill=lightgray] (-1,-1) rectangle (1,1);
		\end{scope}
		
		\begin{scope}[shift={(2.6,1)},scale=1,rotate=0]
		    \draw[very thick, fill=lightgray] (-1,-1) rectangle (1,1);
		\end{scope}
		
		\begin{scope}[shift={(-1.24,1.24)},scale=0.943,rotate=21.77]
		    \draw[very thick, fill=lightgray] (-1,-1) rectangle (1,1);
		\end{scope}

		
		
	\end{tikzpicture}
	\end{center}
	\caption{(Left): The square packing with $2n-1$ contacts described in \Cref{p:sim} for $n=7$.
	(Right): A square packing with $2n-1$ contacts which can be formed from the square packing on the left by perturbing the values of the radii.
	All the squares except the top left square maintain their original orientation,
	while the top left square simply needs to rotate slightly to maintain the necessary $2n-1$ contacts.}
	\label{fig:1sqrotated}
\end{figure}
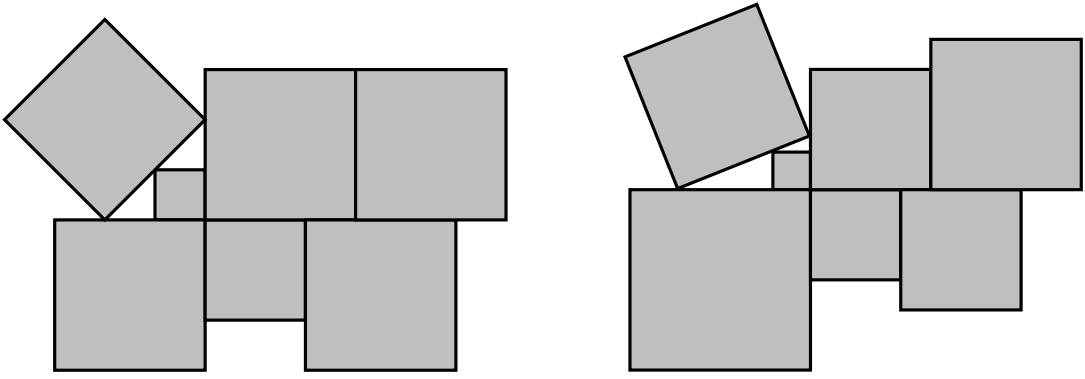

Because of \Cref{p:sim}, we shall restrict which type of contacts we are interested in.
Let $P = \{S_1,\ldots,S_n\}$ be a square packing with $S_i = r_i R_{\theta_i} S +p_i$ for each $i \in [n]$.
We say that the distinct squares $S_i$ and $S_j$ have a \emph{face-to-face contact} if the set $S_i\cap S_j$ is a line segment $[z,z']$ with $z \neq z'$.
It is important to note that if the vertex pair $\{i,j\}$ describe a face-to-face contact, then $\theta_i = \theta_j$.
Another useful observation is the following: 
if $P$ is a homothetic square packing with contact graph $G=([n],E)$,
then the face-to-face contacts of $P$ are exactly the edges in the symmetric difference of $E_x$ and $E_y$.

\begin{corollary}\label{cor:rotatedsquares}
	Let $r_1,\ldots,r_n$ be positive scalars that satisfy the weak generic condition.
    Then every packing of $n$ squares with radii $r_1,\ldots, r_n$ has at most $2n-2$ face-to-face contacts.
    Furthermore, if a given square packing has $2n-2$ face-to-face contacts, 
    then it is a homothetic square packing with no four squares sharing a corner.
\end{corollary}

\begin{proof}
	Let $P$ be a square packing with radii $r_1,\ldots,r_n$, centres $p_1,\ldots,p_n$ and angles $\theta_1,\ldots,\theta_n$.
	Define the equivalence relation $\sim$ on $[n]$ by setting $i \sim j$ if and only if $\theta_i = \theta_j$,
	and set $\tilde{n}_1,\ldots,\tilde{n}_m$ to be the equivalence classes of $[n]$.
	Each square packing $P(\tilde{n}_i) := \{S_j : j \sim n_i\}$ is homothetic,
	hence each has at most $2|\tilde{n}_i| -2$ contacts by \Cref{mainthm}.
	Since there can be no face-to-face contacts between $P(\tilde{n}_i)$ and $P(\tilde{n}_j)$ when $i \not\sim j$,
	we have that $P$ has at most $\sum_{i=1}^m (2|\tilde{n}_i| -2) = 2n - 2m \leq 2n-2$ face-to-face contacts.
	
	Suppose that $P$ has $2n-2$ face-to-face contacts.
	Then $m=1$ and $P$ is a homothetic square packing with $|E_x \triangle E_y| = 2n-2$.
	As $E \supset E_x \triangle E_y$ and $|E| \leq 2n-2$ (\Cref{mainthm}),
	we have $E = E_x \triangle E_y$,
	i.e., $E_x$ and $E_y$ are disjoint sets.
	Hence $P$ has no four squares sharing a corner,
	as this would imply the intersection $E_x \cap E_y$ is non-empty.
\end{proof}

\section{Failure of the natural analogue of \texorpdfstring{\Cref{mainthm}}{Theorem 1} for homothetic cube packings}\label{sec:cube}

For this section we fix the standard cube to be the set $C := \{ (x,y,z) : -1 \leq x,y,z\leq 1 \}$.
We can analogously define the concept of a \emph{homothetic cube packing} to be a set $P=\{C_1,\ldots,C_n\}$ of homothetic copies of $C$ with pairwise disjoint interior.
Similarly we define the centres $p_1,\ldots, p_n \in \mathbb{R}^3$ and radii $r_1,\ldots,r_n >0$ to be the values such that $C_i = r_i C +p_i$ for each $i \in [n]$,
and define the contact graph $G=([n],E)$ by setting $\{i,j\} \in E$ if and only if $i\neq j$ and $C_i \cap C_i \neq \emptyset$.
Given two cubes $C_i$ and $C_j$ in contact with centres $p_i = (x_i,y_i,z_i)$, $p_j = (x_j,y_j,z_j)$ and radii $r_i,r_j$ respectively,
at least one of the following three possibilities hold:
\begin{enumerate}
	\item $r_i + r_j = |x_i - x_j| \geq \max\{ |y_i - y_j|, |z_i - z_j|\}$, in which case we say $C_i$ and $C_j$ have a \emph{$x$-direction contact},
	\item $r_i + r_j = |y_i - y_j| \geq \max\{ |x_i - x_j|, |z_i - z_j|\}$, in which case we say $C_i$ and $C_j$ have a \emph{$y$-direction contact},
	\item $r_i + r_j = |z_i - z_j| \geq \max\{ |x_i - x_j|, |y_i - y_j|\}$, in which case we say $C_i$ and $C_j$ have a \emph{$z$-direction contact}.
\end{enumerate}
An obvious question now is: can we extend our results for homothetic square packings to homothetic cube packings?
Before begin looking into this question in more detail,
we need to understand what the analogous amount of contacts should be.
If each subgraph of the contact $G=([n],E)$ of $P$ formed by either the $x$-, $y$- or $z$-direction contact edges has no cycles,
then $G$ can have at most $3n-3$ edges.
Hence we would (naively) expect that randomly chosen radii will force the number of contacts to be bounded by $3n-3$.
However it is easy to see that this bound must fail for sufficiently large $n$.

\begin{proposition}\label{p:cube1}
	Let $n \geq 7$ and choose any $n$ positive scalars $r_1, \ldots, r_n$.
    Then there exists a homothetic packing of $n$ cubes with radii $r_1,\ldots,r_n$ and more than $3n-3$ contacts.
\end{proposition}

\begin{proof}
	Let $n = 8k+\ell$ for some $\ell \in \{0,\ldots,7\}$, and choose any set of positive scalars $r_1,\ldots,r_n$.
	Choose some $R>0$ such that $4 r_i < R$ for each $i \in [n]$.
	For each $i \in [n]$,
	define the unique non-negative integers $a_i,b_i$ that are the quotient and remainder of $i$ divided by $8$ respectively,
	i.e., $i = 8a_i +b_i$.
	With this we define
	\begin{align*}
		p_i :=
		\begin{cases}
			(r_i+ a_i R,r_i,r_i) &\text{if } b_i = 0\\
			(r_i+ a_i R,r_i,-r_i) &\text{if } b_i = 1\\
			(r_i+ a_i R,-r_i,r_i) &\text{if } b_i = 2\\
			(r_i+ a_i R,-r_i,-r_i) &\text{if } b_i = 3\\
			(-r_i+ a_i R,r_i,r_i) &\text{if } b_i = 4\\
			(-r_i+ a_i R,r_i,-r_i) &\text{if } b_i = 5\\
			(-r_i+ a_i R,-r_i,r_i) &\text{if } b_i = 6\\
			(-r_i+ a_i R,-r_i,-r_i) &\text{if } b_i = 7.
		\end{cases}
	\end{align*}
	Let $P = \{C_1,\ldots,C_n\}$ be the homothetic cube packing where $C_i = r_i C +p_i$ for each $i \in [n]$.
	Two distinct cubes $C_i,C_j$ are in contact if and only if $a_i = a_j$,
	hence
	\begin{align}\label{eq:cube}
		|E| = \binom{8}{2}k + \binom{\ell}{2} = 3n - 3 + \frac{n + \ell(\ell-8) +6}{2}.
	\end{align}
	If $n \geq 11$ then $|E| > 3n-3$ since $\ell(\ell - 8) + 6 \geq -10$.
	For small values of $n$ we can substitute the corresponding remainder $\ell$ into \cref{eq:cube}:
	if $n =7$ then $\ell = 7$ and $|E| =3n$;
	if $n =8$ then $\ell = 0$ and $|E| =3n +4$;
	if $n =9$ then $\ell = 1$ and $|E| =3n +1$;
	if $n =10$ then $\ell = 2$ and $|E| = 3n -1$.
	Hence $P$ has more than $3n-3$ contacts when $n \geq 7$.
\end{proof}

As the complete graph with $n \leq 6$ vertices has at most $3n-3$ edges,
\Cref{p:cube1} cannot be improved.
However, the homothetic packing of $n$ cubes described in \Cref{p:cube1} will always have less than $12(n/8) < 3n-3$ \emph{face-to-face contacts} (i.e., a contact where the intersection is a 2-dimensional convex set).
This leads one to wonder: is the natural analogue to \Cref{mainthm} true if we restrict to face-to-face contacts?
Unfortunately, this too can fail for very general choices of radii.

\begin{proposition}\label{p:cube2}
    Let $n \geq 4$ and $r_1 \geq \ldots \geq r_n >0$.
    If $r_4 + \ldots + r_n < r_3$,
    then there exists a homothetic packing of $n$ cubes with radii $r_1,\ldots,r_n$ and at least $4n-11$ face-to-face contacts (and hence more than $3n-3$ face-to-face contacts when $n \geq 9$).
    
\end{proposition}

\begin{proof}
    We begin by defining the homothetic cube packing $P = \{C_1,\ldots,C_n\}$ with radii $r_1,\ldots,r_n$ and centres $p_1,\ldots,p_n$.
    Fix
    \begin{align*}
        p_1 = (-r_1,0,-r_1), \qquad p_2 = (r_2,0,-r_2), \qquad p_3 = (r_3 +r_n ,0,r_3). 
    \end{align*}
    Define $s_4 = r_4 - r_3$ and $s_i = r_i + \sum_{k=4}^{i-1} 2 r_k - r_3$ for $i >4$.
    We now fix $p_i = (-r_i + r_n ,s_i,r_i)$ for all $i \geq 4$.
    With this we fix $C_i = r_i C +p_i$ for each $i \in [n]$.
    
    Let $E$ be the set of pairs $\{i,j\}$ where the cubes $C_i$ and $C_j$ have a face-to-face contact.
    Then
    \begin{align*}
        E = \big\{ \{1,2\},\{2,3\} \big\} \cup \big\{ \{i,j\}: i \in \{1,2,3\} , ~ 4 \leq j \leq n \big\} \cup \big\{ \{i,i+1\} :  4 \leq i \leq n-1 \big\},
    \end{align*}
    and hence $P$ has $2 + 3(n-3) + n-4 = 4n -11$ face-to-face contacts.
\end{proof}

\subsection*{Acknowledgement}
The author was supported by the Heilbronn Institute for Mathematical Research and the Austrian Science Fund (FWF): P31888.

\end{document}